\documentclass[10pt]{article} 
\usepackage[utf8]{inputenc}
\usepackage[english]{babel}
\usepackage[]{amsthm} 
\usepackage[]{amssymb} 
\usepackage[]{amsmath}
\usepackage{faktor}
\usepackage{xcolor}
\usepackage{dsfont}
\usepackage[shortlabels]{enumitem}

\usepackage{csquotes}

\usepackage{geometry}
\newcommand{\aaa}{\mathbf{a}}
\newcommand{\bbb}{\mathbf{b}}
\newcommand{\ccc}{\mathbf{c}}
\newcommand{\GR}{\mathrm{GR}}
\newcommand{\SR}{\mathrm{SR}}
\newcommand{\ml}{\mathrm{ml}}
\newcommand{\BC}{\mathbb{C}}
\newcommand{\ot}{\otimes}

\newtheorem{thm}{Theorem}
\newtheorem{lem}[thm]{Lemma}
\newtheorem{prop}[thm]{Proposition}
\newtheorem{cor}[thm]{Corollary}

\theoremstyle{definition}
\newtheorem{defn}[thm]{Definition}
\newtheorem{rem}[thm]{Remark}
\newtheorem{ex}[thm]{Example}

\usepackage{alphalph}

\title{Geometric Rank and Linear Determinantal Varieties}
\author{Runshi Geng
\footnote{Texas A\&M University, ORCID: 0000-0003-3440-5148, runshi.geng@gmail.com}}

\begin{document}
\maketitle
\begin{abstract}
There are close relations between tripartite tensors with bounded geometric ranks and linear determinantal varieties with bounded codimensions. We study linear determinantal varieties with bounded codimensions, and prove upper bounds of the dimensions of the ambient spaces. Using those results, we classify tensors with geometric rank 3, find upper bounds of multilinear ranks of primitive tensors with geometric rank 4, and prove the existence of such upper bounds in general. We extend results of tripartite tensors to $n$-part tensors, showing the equivalence between geometric rank 1 and partition rank 1.
\end{abstract}

\vspace{-0.3em}\hspace*{2.5em}\textbf{Key words:} geometric rank, multilinear rank, linear determinantal variety, tensor.

\vspace{-0.3em}\hspace*{2.5em}2020 Mathematics Subject Classification: 15A69, 68Q17, 14L30.

\section{Introduction}
\subsection{Geometric Rank}
Various types of ranks of tensors have been introduced and studied in numerous areas such as algebraic complexity, extremal combinatorics and quantum information theory. Subrank was introduced by Strassen to study the algebraic complexity of matrix multiplication \cite{str}, and its asymptotic version plays an important role in Strassen's laser method \cite{str87}, which people have utilized to obtain upper bounds of the exponent of matrix multiplication. Slice rank arose in the study of the cap set problem \cite{slice}, and it turned out to be helpful in the study of the sunflower problem \cite{naslund2017upper}. Slice rank and subrank were also studied from the point of view of quantum information theory \cite{universal}. Analytic rank was introduced by \cite{analytic} in the context of Fourier analysis, and \cite{analytic2} showed it lower bounds slice rank and can replace slice rank in the resolution of cap set problem.

Geometric rank was introduced in \cite{KMZ20} as an extension of analytic rank from finite fields to algebraically closed fields, and as a tool to find upper bounds on border subrank and lower bounds on slice rank. \cite{GL20} took a step further studying geometric rank systematically, giving results on tensors with geometric rank at most 3. \cite{CM21} showed that the partition rank is at most $2^{n-1}$ times of the geometric rank for $n$-part tensors. Putting different types of ranks in an increasing order, we have:
\begin{align*}
    \mathrm{Subrank}\leq\mathrm{Border\;Subrank}&\leq\mathrm{Geometric\;Rank}\\
    &\leq\mathrm{Partition\;Rank}\leq\mathrm{Slice\;Rank}\leq\mathrm{Multilinear\;Ranks}\leq\mathrm{Rank}.
\end{align*}
Any tensor $T\in A^{(1)}\otimes\cdots\otimes A^{(n)}:=\BC^{m_1}\ot\cdots\ot\BC^{m_n}$ can be regarded as a multilinear function $T:A^{(1)*}\times \cdots \times A^{(n)*}\rightarrow \BC$. Its \textbf{geometric rank} is defined to be:
$$\GR(T):=\mathrm{codim}\{(x_1,\dots,x_{n-1})\in A^{(1)*}\times \cdots \times A^{(n-1)*}\mid T(x_1,\dots,x_{n-1},x_n)=0,\forall x_n\in A^{(n)*}\}.$$

A tripartite tensor $T\in A\ot B\ot C:=\BC^{\aaa}\ot\BC^{\bbb}\ot\BC^{\ccc}$ can be regarded as a linear map $T_A:A^*\rightarrow B\ot C$. Omitting the subscripts when there is no ambiguity, $T(A^*)\subset B\ot C$ is an $\aaa$-dimensional space of $\bbb\times\ccc$ matrices. Fixing bases $\{a_i\}_{i=1}^{\aaa},\{b_j\}_{j=1}^{\bbb}$ and $\{c_k\}_{k=1}^{\ccc}$ of $A,B$ and $C$, and the dual basis $\{\alpha_i\}_{i=1}^{\aaa}$ of $A^*$ corresponding to $\{a_i\}_{i=1}^{\aaa}$, we often represent $T(A^*)$ by a general point $T(\sum x_i\alpha_i)$ of $T(A^*)$ in a matrix form. That is, $T(A^*)$ will be written as a $\bbb\times\ccc$ matrix whose entries are linear forms in variables $x_i$'s.

Let $A^*_i:=\{\alpha\in A^*\mid \mathrm{rank}(T(\alpha))\leq i\}$, and $B^*_i$ and $C^*_i$ are defined similarly. An alternative definition, proved to be equivalent to the
previous in \cite[Theorem~3.1]{KMZ20}, is:
\begin{equation}\label{altdef}
    \GR(T)=\min\{\mathrm{codim}(A^*_i)+i\}=\min\{\mathrm{codim}(B^*_i)+i\}=\min\{\mathrm{codim}(C^*_i)+i\}
\end{equation}
which shows close relations of geometric rank with spaces of matrices of bounded rank and more generally determinantal varieties.

Given $r\leq \min\{\aaa,\bbb,\ccc\}$, let $\mathcal{GR}_r$ be the set of tensors with geometric rank at most $r$. As $\mathcal{GR}_r$ is Zariski closed \cite{KMZ20}, our goal is to give geometric interpretations of those varieties, and classify the tensors in $\mathcal{GR}_r$ up to changes of bases and permutations of $A,B$ and $C$ if possible.

\subsection{Determinantal Variety}
For a linear space of matrices $E\subset A\ot B:=\BC^{\aaa}\ot \BC^{\bbb}$, let $E_r$ be the locus of matrices of rank at most $r$, for $r\leq \min\{\aaa,\bbb\}$. In other words, $\mathbb{P}E_r=\mathbb{P}E\cap\sigma_r(\mathrm{Seg}(\mathbb{P}A\times\mathbb{P}B))$, the intersection of $\mathbb{P}E$ with the $r$-th secant variety of the Segre variety. $E_r$ is cut out by all $(r+1)\times(r+1)$ minors set theoretically, and is called a \textbf{linear determinantal variety} (see, e.g., \cite[Ch. II]{arbarello1985geometry}).

Let $H:=A\ot B$, then $H_r$ is the affine cone of $\sigma_r(\mathrm{Seg}(\mathbb{P}A\times\mathbb{P}B))$ and is called a generic determinantal variety. The defining ideal $I(H_r)$ is prime and generated by all $(r+1)\times(r+1)$ minors \cite{weyl1946classical}, and $\mathrm{codim}(H_r)=(\aaa-r)(\bbb-r)$ \cite{hochster1971cohen}. Since $E_r=H_r\cap E$ is a linear section of $H_r$, $\mathrm{codim}_E (E_r)\leq (\aaa-r)(\bbb-r)$.

To study $\mathcal{GR}_r$, note that by definition $T(A^*_i)$ consists of matrices in $B\ot C$ of rank at most $i$, so it is a linear determinantal variety. Since $\mathrm{codim}_{T(A^*)}(T(A^*_i))=\mathrm{codim}_{A^*}(A^*_i)$, by (\ref{altdef}) we need to find all linear spaces $E\subset B\ot C$ satisfying $\mathrm{codim}_E (E_i)\leq r-i$ for $0\leq i\leq r$.

\subsection{Space of Matrices of Bounded Rank}
A linear space of matrices $E\subset A\ot B:=\BC^{\aaa}\ot \BC^{\bbb}$ is said to have \textbf{bounded rank r} if all matrices in $E$ have rank at most $r$, i.e., $E_r=E$. There are two important classes of spaces of bounded rank -- primitive spaces \cite{Atk81} and compression spaces \cite{EH88}. $E$ is \textbf{compression} if there exist $A'\subset A$ and $B'\subset B$ of dimension $p$ and $q$, such that $E\subset A'\ot B+A\ot B'$ and $p+q=r$. $E$ is \textbf{primitive} if for any subspaces $A'\subset A$ or $B'\subset B$ of codimension 1, $E\not\subset A'\ot B$ or $A\ot B'$, and neither $E\cap(A'\ot B)$ nor $E\cap(A\ot B')$ has bounded rank $r-1$.


Atkinson and Lloyd showed that every space of bounded rank $r$ that is not compression equals to a "sum" of compression space of bounded rank $i$ and a primitive space of bounded rank $r-i$ for some $i$ in \cite{Atk81}. Later all primitive spaces of bounded rank 2 and 3 were classified in \cite{Atk83}. \cite{EH88} recasted the study with sheaves and gave geometric interpretations of all primitive spaces of bounded rank 3 as matrices.

The alternative definition (\ref{altdef}) shows $\GR(T)\leq r$ if at least one of $T(A^*), T(B^*)$ and $T(C^*)$ has bounded rank $r$. In fact, when $r=1$ and $2$ this condition is necessary \cite{GL20}. But it fails to be necessary when $r=3$ as there are two exceptions (see Theorem \ref{gr3thm}).

\subsection{Matrix Multiplication Tensor}
In the study of arithmetic complexity of matrix multiplication, Strassen found that the number of additions and multiplications are required to multiply two matrices asymptotically is determined by the rank of matrix multiplication tensors \cite{str87}. 

For positive integers $e\leq h\leq l$, put $A=\BC^{e\times h},B=\BC^{h\times l}$ and $C=\BC^{l\times e}$. Then the \textbf{matrix multiplication tensor} $M_{\langle e,h,l \rangle}$ is defined by $M_{\langle e,h,l \rangle}(x,y,z)=\mathrm{Tr}(xyz)$ for $x\in A^*$, $y\in B^*$ and $z\in C^*$. We often write $M_{\langle n\rangle}:=M_{\langle n,n,n\rangle}$. With proper choices of bases, $M_{\langle e,h,l \rangle}$ may be written as the block form:
\begin{equation}\label{mm}
M_{\langle e,h,l \rangle}(A^*)=
\begin{pmatrix}
D & & &\\
&D&&\\
&&\ddots &\\
&&&D
\end{pmatrix}
\end{equation}
where $D$ is a $e\times h$ block consisting of linearly independent entries and there are $l$ copies of $D$ in $M_{\langle e,h,l \rangle}(A^*)$.

Strassen gave a lower bound of the border subrank of $M_{\langle e,h,l \rangle}$, which is $eh-\lfloor (e+h-l)^2/4\rfloor$ if $e+h\geq l$ and $eh$ otherwise \cite{str87}. And recently \cite{KMZ20} surprisingly found that the above lower bound equals to the geometric rank of $M_{\langle e,h,l \rangle}$, and consequently equals to the border subrank of $M_{\langle e,h,l \rangle}$ since geometric rank upper bounds border subrank.

\subsection{Main Results}
For $T\in A\otimes B\otimes C=\BC^{\aaa}\otimes\BC^{\bbb}\otimes\BC^{\ccc}$, the $\textbf{multilinear ranks}$ are $\ml_A(T):=\mathrm{rank}(T_A)$, $\ml_B(T):=\mathrm{rank}(T_B)$ and $\ml_C(T):=\mathrm{rank}(T_C)$. And the \textbf{slice rank} is $\SR(T):=\min\{\ml_A(T_1)+\ml_B(T_2)+\ml_C(T_3)\mid T=T_1+T_2+T_3\}$.

\begin{defn}\label{primitivedef}
$T$ is \textbf{compression of geometric rank r} if $\GR(T)=\SR(T)=r$. $T$ is \textbf{primitive of geometric rank r} if it cannot be written as $T=X+Y$ with $\GR(X)=r-1$ and $\GR(Y)=1$.
\end{defn}

For tripartite tensors $T\in A\ot B\ot C$, our main results are:
\begin{itemize}
    \item \textbf{Theorem \ref{gr3thm}.} 
    A tensor $T\in A\ot B\ot C$ has geometric rank at most 3 if and only if one of the following conditions holds:
    \begin{enumerate}
        \item $T(A^*)$, $T(B^*)$ or $T(C^*)$ is of bounded rank $3$, or
        \item $\mathrm{SR}(T)\leq 3$, or
        \item up to changes of bases $T=M_{\langle 2\rangle}$.
    \end{enumerate}
    
    If $T$ is primitive of geometric rank 3, then up to changes of bases and permutations of $A$, $B$ and $C$, it is either the matrix multiplication tensor $M_{\langle 2\rangle}$ or the tensor such that $T(A^*)$ is a space of $4\times 4$ skew-symmetric matrices of dimension $4,5$ or $6$.

    \item \textbf{Theorem \ref{gr4thm}.} If $T$ is primitive of geometric rank 4, then either at least 2 of $\ml_A(T)$, $\ml_B(T)$ and $\ml_C(T)$ are at most 6, or all of them are at most 8.
    
    \item \textbf{Theorem \ref{grrthm}.} For all $r$, there exists a positive integer $N_r$, such that if $T$ is primitive of geometric rank $r$, then at least 2 of $\ml_A(T)$, $\ml_B(T)$ and $\ml_C(T)$ are at most $N_r$.
\end{itemize}

For $n\geq 3$ and $n$-part tensors $T\in A^{(1)}\otimes\cdots\otimes A^{(n)}$, we have:
\begin{itemize}
    \item \textbf{Proposition \ref{prop:npart2}.} For $r< n$, $\mathrm{GR}(T)\leq r$ if and only if there exists $i$ such that $T(A^{(i)*})$ has bounded geometric rank $r$ as a space of $(n-1)$-part tensors.
    
    \item \textbf{Proposition \ref{prop:npart3}.} $T$ has geometric rank 1 if and only if it has partition rank 1.
\end{itemize}

Although we assume all tensors are defined over complex field, all results from this paper hold for any algebraically closed field with characteristic zero.

\subsection{Overview}
We begin with discussion of primitive and compression tensors in section \S \ref{section:primitive}. Lemma \ref{primitivelem} gives a criterion to determine if a tensor is primitive, and Corollary \ref{MMprimitive} shows the matrix multiplication tensors are either primitive or compression. Lemma \ref{decomplem} shows any tensor with degenerate geometric rank can be decomposed as a sum of a primitive tensor and a compression tensor.

In section \S \ref{section:determinantal} we study the subspaces $E\subset A\ot B$ whose determinantal varieties $E_k$ have bounded codimensions, especially finding the upper bounds of the dimensions of $A$ and $B$ when $E$ is concise. Proposition \ref{E2codim1prop} gives the classification of spaces whose $3\times 3$ minors all have a common quadratic factor. Proposition \ref{Ercodimnprop} proves the existence of the upper bounds on the dimensions of $A$ and $B$ in general. In our study of determinantal varieties, we observed an error in Proposition 1 of \cite{beauville2018introduction}, see Remark \ref{error} for details.

Using the results on linear determinantal varieties, in section \S \ref{sec:gr3} and \ref{sec:gr4} we conclude the classification of tensors in $\mathcal{GR}_3$ in Theorem \ref{gr3thm}, find upper bounds of multilinear ranks of primitive tensors with geometric rank 4 in Theorem \ref{gr4thm}, and obtain the existence of such upper bounds for tensors with bounded geometric rank in general.

In sections \S \ref{section:npart} we shift our study from tripartite tensors to $n$-part tensors. Proposition \ref{altdef2} generalizes the alternative definition (\ref{altdef}). Proposition \ref{prop:npart2} shows that $N$-part tensors with small geometric ranks always correspond to spaces of $(N-1)$-part tensors of bounded geometric ranks. Finally we conclude the equivalence between partition rank 1 and geometric rank 1 in Proposition \ref{prop:npart3}.

\section*{Acknowledgements}
I appreciate my advisor Joseph Landsberg for massive instructions on my research on geometric ranks, and lots of comments and corrections to this paper. I also thank Giorgio Ottaviani for useful conversations, Guy Moshkovitz for useful questions, and the anonymous referee for suggestions and corrections. 

\section{Primitive and Compression Tensors}\label{section:primitive}
The following lemma gives a direct way to determine whether a tensor is primitive in general.

\begin{lem}\label{primitivelem}
Given $T$ with $1<\GR(T)=r<\SR(T)$, then $T$ is not primitive if and only if $\exists i< r$ such that by a permutation of $A,B$ and $C$, $\mathrm{codim}(A^*_i)=r-i$ and $A^*_i$ has a component of maximal dimension that is contained in a hyperplane of $A^*$.
\end{lem}
\begin{proof}
Let $\{a_i\}_{i=1}^{\aaa}$ be a basis of $A$, and $\{\alpha_i\}_{i=1}^{\aaa}$ be the dual basis of $A^*$. Write $A':=\langle a_2,\cdots,a_\aaa\rangle$, so $A'^*=\langle \alpha_2,\cdots,\alpha_\aaa\rangle$.

($\Rightarrow$) $T$ is not primitive if and only if we can decompose $T=X+Y$ with $\GR(X)=r-1$ and $\GR(Y)=1$. Since $\GR(Y)=1$ if and only if $\SR(Y)=1$,  by permuting $A,B$ and $C$ assume $\ml_A(Y)=1$, and by changing basis of $A$ assume $Y\in \langle a_1\rangle\ot B\ot C$. 

Then $T=X'+Y'$ where $X':=T|_{A'\ot B\ot C}$ and $Y':=T|_{\langle a_1\rangle\ot B\ot C}$. Since $X'=X|_{A'\ot B\ot C}$, $\GR(X')\leq \GR(X)=r-1$. By subadditivity of geometric rank and $\SR(Y')=\GR(Y')=1$, $\GR(X')=r-1$. By (\ref{altdef}) there exists $i\leq r-1$ such that $\mathrm{codim}\{\alpha\in A^*\mid  \mathrm{rank}(X'(\alpha))\leq i\}\leq r-1-i$. Then $\{\alpha\in A^*\mid  \mathrm{rank}(X'(\alpha))\leq i\}\cap A'^*\subset A^*_i$ has codimension $r-i$ in $A^*$ and is contained in a hyperplane.

($\Leftarrow$) Assume $\mathrm{codim}(A^*_i)=r-i$ and $A^*_i$ has a component $Z$ of maximal dimension contained in $A'^*$. Let $X'$ and $Y'$ be defined the same as above. By definition $\{\alpha\in A'^*\mid  \mathrm{rank}(X'(\alpha))\leq i\}\supset Z$ so has codimension at most $r-i$ in $A^*$, then its codimension is at most $r-1-i$ in $A'^*$. Since $X'\in A'\ot B\ot C$, $\GR(X')\leq r-1$. By $T=X'+Y'$ and subadditivity of geometric rank, $\GR(X')=r-1$ and $\GR(Y')=1$.
\end{proof}

\begin{cor}\label{MMprimitive}
For positive integers $e\leq h\leq l$, $M_{\langle e,h,l \rangle}$ is primitive if $e\geq 2$ and $e+h\geq l$, and it is compression otherwise.
\end{cor}
\begin{proof}
By Theorem 6.1 of \cite{KMZ20}, $\GR(M_{\langle e,h,l \rangle})=eh$ if $e+h\leq l$ or $e=1$. Since $\GR(M_{\langle e,h,l \rangle})\leq \SR(M_{\langle e,h,l \rangle})\leq \ml_A(M_{\langle e,h,l \rangle})=eh$, we have $\GR(M_{\langle e,h,l \rangle})=\SR(M_{\langle e,h,l \rangle})=eh$ and therefore $M_{\langle e,h,l \rangle}$ is compression.

Assume $e\geq 2$ and $e+h\geq l$. The component of the maximal dimension $Z\subset A_i$ is determined by all $k\times k$ minors of $D$, where $k=\min\{e,\lceil \frac{i+1}{l}\rceil\}$. By \cite[Theorem~2.1]{Eis88}, $\mathrm{codim}(A_i)=\mathrm{codim}(Z)=(e+1-k)(h+1-k)$. So (\ref{altdef}) achieves minimum only at $i=\lceil\frac{e+h-l}{2}\rceil l$ and $\lfloor\frac{e+h-l}{2}\rfloor l$. Then $k>1$ and $Z$ is not contained in any hyperplane.
\end{proof}

Although we define the primitive and compression tensors as analogues of primitive and compression spaces of matrices, their relations are subtle.

By definition $T$ is compression of $\GR(T)=r$ if at least one of $T(A^*),T(B^*)$ or $T(C^*)$ is a compression space of bounded rank $r$ and none has bounded rank $r-1$. The converse is true only for $r\leq 2$, as $T:=\sum_{i=1}^{m}(a_1\ot b_i\ot c_i+a_i\ot b_1\ot c_i+a_i\ot b_i\ot c_1)$ is compression of $\GR(T)=3$ but $T(A^*),T(B^*)$ and $T(C^*)$ contain elements of full rank.

If $T$ is primitive of $\GR(T)=r$ and $T(A^*)$ has bounded rank $r$, then $T(A^*)$ is primitive of bounded rank $r$ (after deleting zero rows and columns). Similarly for $T(B^*)$ and $T(C^*)$. However $T$ could be primitive when $T(A^*),T(B^*)$ and $T(C^*)$ do not have bounded rank $r$. 

For example, $M_{\langle 2\rangle}$ is primitive of geometric rank 3 by Corollary \ref{MMprimitive}. But since $M_{\langle 2\rangle}(A^*)$ can be written as the block diagonal form (\ref{mm}), generic matrices in $M_{\langle 2\rangle}(A^*)$ have full rank 4. Therefore $M_{\langle 2\rangle}(A^*)$ does not have bounded rank 3. For the same reason, $M_{\langle 2\rangle}(B^*)$ and $M_{\langle 2\rangle}(C^*)$ do not either.

There is no primitive space of bounded rank 1, and all primitive spaces bounded rank $2$ and $3$ are listed in \cite{Atk83,EH88}. We check every such primitive space and conclude that for $r\leq 3$, if $T(A^*)$ is primitive of bounded rank $r$, then $T$ is primitive of geometric rank $r$. It is not known if this property persists when $r>3$, because the set of all primitive spaces of larger bounded rank are not classified yet.

\begin{lem}\label{decomplem}
If $T$ is not compression (i.e., $\GR(T)<\SR(T)$), then there exist a primitive tensor $T_p$ and a compression tensor $T_c$, such that $T=T_p+T_c$ and $\GR(T_p)+\GR(T_c)=\GR(T)$.
\end{lem}
\begin{proof}
If $T$ is primitive, set $T_p=T$ and $T_c=0$.

If $T$ is not primitive, assume $\GR(T)=r$, then we can write $T=X_1+Y_1$ such that $\GR(X_1)=r-1$ and $\GR(Y_1)=1$. Similarly, whenever $X_i$ is not primitive or zero, we can write $X_i=X_{i+1}+Y_{i+1}$ such that $\GR(X_i)=r-i$ and $\GR(Y_1)=1$. If all $X_i$'s obtained this way are not primitive, we have a decomposition $T=Y_1+\cdots+Y_r$ where each $Y_i$ has geometric rank 1 so has slice rank 1. This implies $\mathrm{SR}(T)=r=\mathrm{GR}(T)$, contradicting the assumption $\GR(T)<\SR(T)$. 

So there exists $n<r$ such that $X_n$ is primitive, then we obtain $T=T_p+T_c$ where $T_p:=X_n$ and $T_c:=Y_1+\cdots+Y_n$. Since $\GR(T_p)=r-n$ and $\sum\GR(Y_i)=\sum\SR(Y_i)=n$, by subadditivity of geometric rank and slice rank, $\GR(T_c)=\SR(T_c)=n$. Therefore $T_c$ is compression.
\end{proof}

\begin{ex}[Above decomposition is not unique]
Let $T\in A\ot B\ot C=\BC^5\ot\BC^5\ot\BC^6$ be defined as
\begin{align*}
T:=&a_1\ot(b_2\ot c_1+b_3\ot c_2+b_4\ot c_3)+a_2\ot(b_1\ot c_1 -b_3\ot c_4 -b_4\ot c_5)\\
&+a_3\ot(b_1\ot c_2+b_2\ot c_4 -b_4\ot c_6)+a_4\ot(b_1\ot c_3 + b_2\ot c_5 +b_3\ot c_6)+a_5\ot b_5\ot c_6
\end{align*}
where $\{a_i\}_{i=1}^{5},\{b_j\}_{j=1}^{5}$ and $\{c_k\}_{k=1}^{6}$ are bases of $A,B$ and $C$ respectively. So
$$T(A^*)=\begin{pmatrix}
x_2 & x_3 & x_4 & 0 & 0 & 0\\
x_1 & 0 & 0 & x_3 & x_4 & 0\\
0 & x_1 & 0 & -x_2 & 0 & x_4\\
0 & 0 & x_1 & 0 & -x_2 & -x_3\\
0 & 0 & 0 & 0 & 0 & x_5
\end{pmatrix}.$$

Let $X_1:=T|_{A\ot B\ot \langle c_1,\cdots,c_5\rangle},Y_1:=T|_{A\ot B\ot \langle c_6\rangle}, X_2:=T|_{A\ot \langle b_1,\cdots,b_4\rangle\ot C}$ and $Y_2:=T|_{A\ot \langle b_5 \rangle \ot C}$. Since $X_1(A^*)$ consists of the first 5 columns of $T(A^*)$ and $X_2(A^*)$ consists of the first $4$ rows of $T(A^*)$, they are primitive spaces of bounded rank 3 (after deleting the zero columns and rows). So $X_1$ and $X_2$ are primitive of geometric rank $3$, and $T=X_1+Y_1=X_2+Y_2$ gives two different decompositions satisfying the conditions in Lemma \ref{decomplem}.
\end{ex}

By Lemma \ref{decomplem}, to classify the set of tensors of geometric rank at most $r$, it suffices to find all primitive tensors of geometric rank at most $r$. In terms of these notations, the classification of tensors of geometric rank at most 1 and 2 from \cite[Remark~2.6, Theorem~3.1]{GL20} can be rephrased as: 
\begin{itemize}
    \item There are no primitive tensors of geometric rank 1.
    \item The only primitive tensor of geometric rank 2 is (up to changes of bases) the skew-symmetric $3\times3\times3$ tensor.
\end{itemize}

\section{Determinantal Varieties of Bounded Codimensions}\label{section:determinantal}
Let $E\subset\mathbb{C}^{\aaa}\otimes\mathbb{C}^{\bbb}=:A\otimes B$ be a linear subspace of dimension $\ccc$. Fix a basis $\{e_i,1\leq i \leq \ccc\}$ of $E$ and bases of $A$ and $B$, then each $e_i$ can be written as an $\aaa\times\bbb$ matrix. Similar to how we represent $T(A^*)\subset B\ot C$ in Section 1.1, $E$ is represented by the matrix corresponding to a general point $\sum_i x_ie_i$ of $E$, i.e., $E=(y^i_j)_{1\leq i\leq \aaa,1\leq j\leq \bbb}$, where each $y^i_j$ is a linear form in the variables $x_1,\cdots,x_{\ccc}$. For two subspaces $F,F'\subset E$, let $F+F'$ denote the sum of the two corresponding matrices of linear forms. 

Denote the $(i_1,\cdots,i_k)\times(j_1,\cdots,j_k)$ minor of $E$ as $\Delta^{i_1,\cdots,i_k}_{j_1,\cdots,j_k}$ and $\Delta_k:=\Delta^{12\cdots k}_{12\cdots k}$. Unless otherwise stated, the codimension of a subset always refers to the codimension in $E$ or $\mathbb{P}E$. 

\subsection{Case $\mathrm{codim}(E_r)=1$}
This subsection studies the case $\mathrm{codim}(E_r)=1$, i.e. all nonzero $(r+1)\times(r+1)$ minors of $E$ has a common polynomial factor of degree at least 1. 

Lemma \ref{codim1lem} is a more detailed version of Lemma 6.4 of \cite{GL20}, and Lemma \ref{codim1lem2} generalizes Lemma 6.5 of \cite{GL20}.

\begin{lem}\label{codim1lem}
Let $E\subset\mathbb{C}^{\aaa}\otimes\mathbb{C}^{\bbb}$, $r<\aaa,\bbb$ and $E_r\neq E$. If there exists a degree $r+1$ polynomial $P$ dividing all $(r+1)\times(r+1)$ minors of $E$, then either $P$ factors into a product of linear forms, or $E\subset\mathbb{C}^{r+1}\otimes\mathbb{C}^{r+1}$.
\end{lem}

\begin{proof}
The hypothesis that all $(r+1)\times (r+1)$ minors of $E$ are equal up to scale is invariant under changes of bases in $A$ and $B$, so we are allowed to perform invertible row and column operations.

Since $E_r\neq E$, there exists a nonzero $(r+1)\times (r+1)$ minor of $E$. By changes of bases we can assume $\Delta_{r+1}=P$. We further assume $\Delta_{r},\cdots,\Delta_2,y^1_1$ are nonzero.

Write $E=(y^i_j)_{1\leq i\leq \aaa,1\leq j\leq \bbb}$. Consider the the block consisting of the first $r+1$ rows and the first $r+2$ columns:
$$\begin{pmatrix}
y^1_1 & \cdots & y^1_{r+1} & y^1_{r+2} \\
\vdots &      & \vdots & \vdots \\
y^{r+1}_1 & \cdots & y^{r+1}_{r+1} & y^{r+1}_{r+2} \\
\end{pmatrix}.$$

Let $I:=(1,2,\cdots,r+1)$. For $j\leq r+1$, expand the minor consisting all columns except the $j$-th along the last column, then we have
$$c_jP=\Delta^{I}_{I\backslash j, r+2}=\sum_{i=1}^{r+1}(-1)^{i+(r+2)-1}y^i_{r+2}\Delta^{I\backslash i}_{I\backslash j}$$
for some $c_j\in\mathbb{C}$. Thus,
$$\begin{pmatrix}
c_1\\
\vdots\\
c_{r+1}
\end{pmatrix}P=
(-1)^{r+1}\begin{pmatrix}
(-1)^{i}\Delta^{I\backslash i}_{I\backslash j}
\end{pmatrix}_{j,i=1}^{r+1}
\begin{pmatrix}
y^1_{r+2}\\
\vdots\\
y^{r+1}_{r+2}
\end{pmatrix}.$$
For every $j\leq r+1$, multiply $(-1)^j$ to the $j$-th row,
\begin{equation}\label{cofactor}
    (-1)^{r+1}\begin{pmatrix}
(-1)^1 c_1\\
\vdots\\
(-1)^{r+1} c_{r+1}
\end{pmatrix}P=
\begin{pmatrix}
(-1)^{i+j}\Delta^{I\backslash i}_{I\backslash j}
\end{pmatrix}_{j,i=1}^{r+1}
\begin{pmatrix}
y^1_{r+2}\\
\vdots\\
y^{r+1}_{r+2}
\end{pmatrix}.
\end{equation}
Now $((-1)^{i+j}\Delta^{I\backslash i}_{I\backslash j})_{j,i=1}^{r+1}$ is the cofactor matrix of the transpose of $(y^i_j)_{i,j=1}^{r+1}$, whose determinant is $\Delta_{r+1}=P$ by assumption. So
$$(-1)^{r+1}
\begin{pmatrix}
y^1_1&\cdots&y^1_{r+1}\\
\vdots & &\vdots\\
y^{r+1}_1&\cdots&y^{r+1}_{r+1}
\end{pmatrix}
\begin{pmatrix}
-c_1\\
\vdots\\
(-1)^{r+1} c_{r+1}
\end{pmatrix}=
\begin{pmatrix}
y^1_{r+2}\\
\vdots\\
y^{r+1}_{r+2}
\end{pmatrix}.$$
Therefore the column vector $(y^1_{r+2},\dots,y^{r+1}_{r+2})^t$ is a linear combination
of all column vectors appearing in the upper left $(r+1)\times (r+1)$ block of $E$, i.e.  $(y^1_{j},\cdots,y^{r+1}_{j})^t, 1\leq j\leq r+1$. By adding linear combinations
of the first $r+1$ columns to the $(r+2)$-th, we may make the first $r+1$ entries of the $(r+2)$-th column equal to zero. Similarly, we may make
the all last $\bbb-r-1$ entries in the first $r+1$ rows equal to zero. By the same argument, we may do 
the same for the first $r+1$ columns. Then the matrix $E$ becomes:
\begin{equation}\label{E'}
E'=\begin{pmatrix}
y^1_1 & \cdots & y^1_{r+1} & 0 & \cdots & 0\\
\vdots &      & \vdots & \vdots &  &\vdots\\
y^{r+1}_1 & \cdots & y^{r+1}_{r+1} & 0 & \cdots & 0\\
0 & \cdots & 0 & \Tilde{y}^{r+2}_{r+2} & \cdots & \Tilde{y}^{r+2}_{\bbb}\\
\vdots &  & \vdots & \vdots &   & \vdots\\
0 & \cdots & 0 & \Tilde{y}^{\aaa}_{r+2} & \cdots & \Tilde{y}^{\aaa}_{\bbb}
\end{pmatrix}.
\end{equation}

If $\Tilde{y}^{r+1+i}_{r+1+j}=0,\forall i,j>0$, let $A'$ be the space corresponding to the first $r+1$ rows of $E'$ and $B$ the first $r+1$ columns, then $E\subset A'\otimes B'=\mathbb{C}^{r+1}\otimes\mathbb{C}^{r+1}$.

If there exists a nonzero $\Tilde{y}^{r+1+i}_{r+1+j}$, by changes of bases assume it is $\Tilde{y}^{r+2}_{r+2}$. For $1\leq i_1<\cdots<i_{r}\leq r+1$, $1\leq j_1<\cdots<j_{r}\leq r+1$, the $(r+1)\times (r+1)$ minor $\Delta^{i_1,\cdots,i_{r},r+2}_{j_1,\cdots,j_{r},r+2}=\Delta^{i_1,\cdots,i_{r}}_{j_1,\cdots,j_{r}}\Tilde{y}^{r+2}_{r+2}$ is a multiple of $\Delta_{r+1}$. Hence all $r\times r$ minors of the upper left $(r+1)\times (r+1)$ block equal up to scale. 

By assumption $\Delta_{r}\neq 0$. Adding a linear combination of the first $r$ columns to the $(r+1)$-th column and a linear combination of the first $r$ rows to the $(r+1)$-th row, we can set all entries in $(r+1)$-th column and row zero except the $(r+1,r+1)$-th entry. Since $\Delta_{r+1}\neq 0$, the $(r+1,r+1)$-th entry is nonzero, written as $\Tilde{y}^{r+1}_{r+1}$. Then $E'$ becomes:
$$E''=\begin{pmatrix}
y^1_1 & \cdots & y^1_{r}&0 & 0 & \cdots & 0\\
\vdots &       &\vdots&0 & \vdots &  &\vdots\\
y^{r}_1 & \cdots & y^{r}_{r} &0 &0 & \cdots & 0\\
0 & \cdots & 0 & \Tilde{y}^{r+1}_{r+1}& 0 & \cdots & 0\\
0 & \cdots & 0 & 0 & \Tilde{y}^{r+2}_{r+2} & \cdots & \Tilde{y}^{r+2}_{\bbb}\\
\vdots &  & \vdots & 0 & \vdots &   & \vdots\\
0 & \cdots & 0 & 0 & \Tilde{y}^{\aaa}_{r+2} & \cdots & \Tilde{y}^{\aaa}_{\bbb}
\end{pmatrix}.$$

Repeat the above process on the upper left $k\times k$ blocks consecutively for $k=r-1,r-2,\cdots,2$, then $E''$ becomes:
$$\begin{pmatrix}
y^1_1 &  &  &  &  &  &\\
 &  \Tilde{y}^2_2 &  &  &  &  &\\
 & & \ddots &  &  &  &\\
 & & &\Tilde{y}^{r+1}_{r+1}&  &  &\\
 & & & & \Tilde{y}^{r+2}_{r+2} & \cdots & \Tilde{y}^{r+2}_{\bbb}\\
 & & & & \vdots &   & \vdots\\
 & & & & \Tilde{y}^{\aaa}_{r+2} & \cdots & \Tilde{y}^{\aaa}_{\bbb}\\
\end{pmatrix}.$$

Therefore $\Delta_{r+1}=y^1_1\Tilde{y}^2_2\cdots\Tilde{y}^{r+1}_{r+1}$ which factors into a product of linear forms.
\end{proof}

\begin{lem}\label{codim1lem2}
Let $E\subset\mathbb{C}^{\aaa}\otimes\mathbb{C}^{\bbb}$, $1\leq r\leq\min\{\aaa,\bbb\}-2$ and $E\neq E_{r+1}$. If there exists a polynomial $P$ of degree $k$ dividing all $(r+1)\times(r+1)$ minors, then:
\begin{enumerate}[(1)]
    \item if $k>r/2+1$ and for any nonzero $(r+1)\times(r+1)$ minor $\Delta$, $P$ and $\Delta/P$ are coprime, then $P$ is a product of linear forms;
    
    \item if $r$ is even, $k=r/2+1$ and for any nonzero $(r+1)\times(r+1)$ minor $\Delta$, $P$ and $\Delta/P$ are coprime, then either $P$ is a product of linear forms or $E\subset \mathbb{C}^{r+2}\otimes \mathbb{C}^{r+2}$;
    
    \item if $r\geq3$ is odd, $k=(r+1)/2$ and $P$ is irreducible, then either $E\subset \mathbb{C}^{r+2}\otimes \mathbb{C}^{\bbb}$, $\mathbb{C}^{\aaa}\otimes \mathbb{C}^{r+2}$, $\mathbb{C}^{r+3}\otimes \mathbb{C}^{r+3}$, or up to changes of bases $E$ has a nonsingular $(r+1)\times(r+1)$ block such that all $r\times r$ minors of it are multiples of $P$. 
\end{enumerate}
\end{lem} 

\begin{proof}
(1) and (2): Proof by induction on $r$. The base case $r=1$ is trivial. Assume $r>1$ and assume that (1) and (2) holds for all integers smaller than $r$.

Given any nonzero $(r+2)\times(r+2)$ minor of $E$, by changes of bases we can assume it is $\Delta_{r+2}$, and we further assume $\Delta_{r+1}\neq 0$.

Write $\Delta_{r+1}=:PQ$ and for $j\leq r+1$, $\Delta^{I}_{I\backslash j, r+2}=:PQ_j$, where each of the polynomials $Q$ and $Q_j$'s either is zero or has degree $r+1-k$. Then similar to Lemma \ref{codim1lem}, we have
$$
(-1)^{r+1}\begin{pmatrix}
-Q_1\\
\vdots\\
(-1)^{r+1} Q_{r+1}
\end{pmatrix}P=
\begin{pmatrix}
(-1)^{i+j}\Delta^{I\backslash i}_{I\backslash j}
\end{pmatrix}_{j,i=1}^{r+1}
\begin{pmatrix}
y^1_{r+2}\\
\vdots\\
y^{r+1}_{r+2}
\end{pmatrix}
.$$
Using the cofactor matrix, we obtain:
$$\frac{(-1)^{r+1}}{Q}
\begin{pmatrix}
y^1_1&\cdots&y^1_{r+1}\\
\vdots & &\vdots\\
y^{r+1}_1&\cdots&y^{r+1}_{r+1}
\end{pmatrix}
\begin{pmatrix}
-Q_1\\
\vdots\\
(-1)^{r+1} Q_{r+1}
\end{pmatrix}=
\begin{pmatrix}
y^1_{r+2}\\
\vdots\\
y^{r+1}_{r+2}
\end{pmatrix}.$$

By adding a rational combination (where the coefficients are $(-1)^{j}Q_j/Q$'s) of the first $r+1$ columns to the $(r+2)$-th column, we can put the first $r+1$ entries of the $(r+2)$-th column zero. By the same argument, put the first $r+1$ entries of the last $\bbb-r-1$ columns zero. And we can do the similar rational row operations to eliminate first $r+1$ entries of the last $\aaa-r-1$ rows. Then $E$ becomes $E'$ of the form (\ref{E'}). 

Since the $(1,\cdots,r+1,r+2)\times(1,\cdots,r+1,r+2)$ minor is not changed by adding rational multiples of the first $r+1$ rows and columns to the $(r+2)$-th row and $(r+2)$-th column respectively, $\Tilde{y}^{r+2}_{r+2}=\frac{\Delta_{r+2}}{\Delta_{r+1}}$. On the other hand, $\Tilde{y}^{r+2}_{r+2}$ has the form $T/Q$ for some polynomial $T$ of degree $k+1$ if not zero, because all coefficients appearing in the row and column operations above are $(-1)^{j}Q_j/Q$'s. Thus,
\begin{equation}\label{ytilde}
    \frac{T}{Q}=\Tilde{y}^{r+2}_{r+2}=\frac{\Delta_{r+2}}{\Delta_{r+1}}=\frac{\Delta_{r+2}}{PQ}=\frac{(\Delta_{r+2}/P)}{Q}
\end{equation}
and $T=\Delta_{r+2}/P$.

Since $P$ and $Q$ are coprime, the fact $P$ divides all $(r+1)\times(r+1)$ minors is preserved after performing the above rational row and column operations. 

If there exists an $r\times r$ minor of the upper left $(r+1)\times(r+1)$ block that is not a multiple of $P$, by changes of bases assume this minor is $\Delta_r$. $P$ divides the minor $\Delta^{1\cdots r,r+2}_{1\cdots r,r+2}=\Tilde{y}^{r+2}_{r+2}\Delta_r=T\Delta_r/Q$, so $T$ is a multiple of $P$. Hence $P^2$ divides $\Delta_{r+2}=TP$. If $k>r/2+1$, $P^2$ has degree $>r+2$, then we must have $\Delta_{r+2}=0$, contradicting to the assumption $\Delta_{r+2}\neq 0$. If $r$ is even and $k=r/2+1$, $\Delta_{r+2}$ is a multiple of $P^2$. By the arbitrariness of the choice of the nonzero $(r+2)\times(r+2)$ minor of $E$, all $(r+2)\times(r+2)$ minors equal to $P^2$ up to scale. By Lemma \ref{codim1lem}, $P$ factors completely or $E\subset\mathbb{C}^{r+2}\otimes\mathbb{C}^{r+2}$. 

If all $r\times r$ minors of the upper left $(r+1)\times(r+1)$ block are multiples of $P$. By induction, apply (1) by replacing $r$ with $r-1$ so $P$ factors into a product of linear forms.

(3): Similar to above let $\Delta_{r+1}=:PQ$ and $\Delta_{r+2}$ are nonzero. Since $P$ is irreducible of degree $k=(r+1)/2$, either $P$ and $Q$ are coprime, or $Q$ equals to $P$ up to scale. In the latter case, we can choose another nonzero $(r+1)\times (r+1)$ minor from the top left $(r+2)\times(r+2)$ block such that $P$ and $Q$ are coprime, unless all $(r+1)\times (r+1)$ minors in the top left $(r+2)\times(r+2)$ block are multiples of $P^2$.

If all $(r+1)\times (r+1)$ minors in the top left $(r+2)\times(r+2)$ block are multiples of $P^2$, applying Lemma \ref{codim1lem} to the top left $(r+2)\times(r+2)$ block we can put $E$ as
$$E'=\begin{pmatrix}
y^1_1 &\cdots&y^1_{r+1} & 0 & y^1_{r+3}& \cdots\\
\vdots &     &\vdots & \vdots & \vdots &  \\
y^{r+1}_1 &\cdots&y^{r+1}_{r+1} & 0 & y^{r+1}_{r+3}& \cdots\\
0 &\cdots&0 & y^{r+2}_{r+2} & y^{r+2}_{r+3}& \cdots\\
y^{r+3}_1&\cdots&y^{r+3}_{r+1} & y^{r+3}_{r+2} & y^{r+3}_{r+3}& \cdots\\
\vdots &     &\vdots & \vdots & \vdots &  \\
\end{pmatrix}.$$

Consider the $(r+1)\times (r+1)$ minors involving $y^{r+2}_{r+2}$ and $r\times r$ minors of the upper left $(r+1)\times (r+1)$ block: $P$ dividing all $(r+1)\times (r+1)$ minors implies $P$ dividing all $r\times r$ minors from the first $r+1$ rows. Apply (2) by replacing $r$ with $r-1$ to the submatrix consisting of the first $r+1$ rows. Since $P$ is irreducible of degree $k>1$, this submatrix is in some $\mathbb{C}^{r+1}\otimes \mathbb{C}^{r+1}$, we can put $y^i_j$ zero for $i\leq r+1$ and $j\geq r+3$ by changing basis of $A$. For the same reason all $y^j_i$ for $i\leq r+1$ and $j\geq r+3$ can be put zero too. Then $E'$ becomes $E''=\mathrm{diag}(B_1,B_2)$ where $B_1$ is a $(r+1)\times (r+1)$ block. If $B_2$ has an nonzero $2\times 2$ minor, consider the $(r+1)\times (r+1)$ minors consisting of it and any $(r-1)\times(r-1)$ minor of $B_1$, applying (1) replacing $r$ with $r-2$ we see $P$ factors into linear forms which contradicts the irreducibility. Therefore $B_2$ has bounded rank 1, then $E\subset \mathbb{C}^{r+2}\otimes \mathbb{C}^{\bbb}$ or $\mathbb{C}^{\aaa}\otimes \mathbb{C}^{r+2}$.

Now assume $P$ and $Q$ are coprime. Similar to the proof above, if there exists an $r\times r$ minor of the upper left $(r+1)\times(r+1)$ block that is not a multiple of $P$, $P^2$ divides $\Delta_{r+2}$. By the arbitrariness of choice of nonzero $(r+2)\times(r+2)$ minors, $P^2$ divides all $(r+2)\times(r+2)$ minors. As $k=(r+1)/2$ and $P$ is irreducible, $P^2$ divides all $(r+2)\times(r+2)$ minors and we can apply (1) by replacing $r$ with $r+1$, then we conclude $E\subset\mathbb{C}^{r+3}\otimes\mathbb{C}^{r+3}$.

Otherwise all $r\times r$ minors of the upper left $(r+1)\times(r+1)$ block are multiples of $P$.
\end{proof}

\begin{cor}\label{codim1cor}
Let $E\subset\mathbb{C}^{\aaa}\otimes\mathbb{C}^{\bbb}$, $1\leq r\leq\min\{\aaa,\bbb\}-2$, and $\mathrm{codim}(E_r)=1$ and $E\neq E_{r+1}$. Then:
\begin{enumerate}
    \item $E_r$ does not contain any irreducible hypersurface of degree $k>r/2+1$;

    \item if $r$ is even and $E_r$ contains an irreducible hypersurface of degree $r/2+1$, then $E\subset \mathbb{C}^{r+2}\otimes \mathbb{C}^{r+2}$.
\end{enumerate}
\end{cor}

\begin{proof}
1. If $E_r$ contains an irreducible hypersurface of degree $k$, there exists an irreducible polynomial $P$ of degree $k$ dividing all $(r+1)\times(r+1)$ minors. Then for any $(r+1)\times(r+1)$ minor $\Delta$, $\Delta/P$ has degree less than $k$ so must be coprime with $P$. By (1) of Lemma \ref{codim1lem2}, $P$ factors, contradicting to the irreducibility.

2. Similar to the proof of 1 except we apply (2) of Lemma \ref{codim1lem2}. As $P$ cannot be a product of linear forms due to irreducibility, we conclude $E\subset\mathbb{C}^{r+2}\otimes \mathbb{C}^{r+2}$.
\end{proof}

\subsection{Case $\mathrm{codim}(E_1)= n$}\label{E1section}
Let $E^{\perp}:=\{f\in A^*\ot B^*\mid f(E)=0\}$. Define the \textbf{index of degeneracy} of $E$ to be one plus the maximum dimension of a linear space contained in $\mathbb{P}E^{\perp}\cap\mathrm{Seg}(\mathbb{P}A^* \times \mathbb{P}B^*)$, denoted as $\kappa$. Equivalently, $\kappa$ is the largest number of entries in the same row or column of $E$ that can be simultaneously put to zero by changing bases of $A$ and $B$. 

The subspace $E$ is called \textbf{E1-generic} if $\kappa=0$. We call this property E1-generic because it corresponds to the notion of 1-generic for spaces of matrices given by Eisenbud \cite{Eis88}, which differs with the notion of 1-generic that is often used for tensors (cf. \cite{1generic}). We list two results of E1-generic spaces of our interest below.

\begin{thm}[Corollary 3.3 and Theorem 2.1 of \cite{Eis88}]\label{1genthm}
Let $m=\min\{\aaa,\bbb\}$. If $E\subset A\otimes B$ is E1-generic, then:
\begin{enumerate}
    \item for $k\leq m-1$, $\mathrm{codim}(E_k)\geq \aaa+\bbb-2k-1$;
    \item if $F\subset E$ is a subspace with $\mathrm{codim}(F)\leq m-1$, then $\mathrm{codim}_F (F_{m-1})=(\aaa-m+1)(\bbb-m+1)$.
\end{enumerate}
\end{thm}

For generic determinantal varieties, i.e. when $E=A\ot B$, one expects $E_k$ has codimension $(\aaa-k)(\bbb-k)$. E1-generic does not means generic but implies the genericity to some extent -- $E_{m-1}$ has the expected codimension, and the codimension of $E_k$ has a lower bound $\aaa+\bbb-2k-1$.

\begin{prop}\label{segreprop}
Let $n:=\mathrm{codim}(E_1)$, then there exist $0\leq j\leq n$ and a linear subspace $F\subset E$ of codimension $j$, such that either $F\subset \mathbb{C}^{k}\otimes\mathbb{C}^{l}$ for some $k+l\leq n+3-j$ and $k,l\geq 2$, or $j=n$ and $F$ has bounded rank 1.
\end{prop}

\begin{proof}
First assume all nonzero $2\times 2$ minor of $E$ are irreducible. So if there is an entry $y^i_j=0$, then either all entries in the $i$-th row or all entries in the $j$-th column are zero. By changes of bases in $A$ and $B$, there exist integers $k,l\geq2$, such that $y^i_j\neq 0$ if and only if $i\leq k, j\leq l$. 
$$E=\begin{pmatrix}
y^1_1 & \cdots & y^1_l & 0 & \cdots & 0\\
\vdots & & \vdots & \vdots & &\vdots\\
y^k_1 & \cdots & y^k_l & 0 & \cdots & 0\\
0 & \cdots & 0 & 0 & \cdots & 0\\
\vdots & & \vdots & \vdots & &\vdots\\
0 & \cdots & 0 & 0 & \cdots & 0\\
\end{pmatrix}.$$

Then the upper left $k\times l$ block of $E$ is E1-generic. By Theorem \ref{1genthm} if $k,l\geq 2$, $\mathrm{codim}(E_1)\geq k+l-3$, so $k+l\leq n+3$.

If there is a $2\times 2$ minor of $M$ that factors into the product of two linear forms $\ell_1,\ell_2$, write $F:=\{\ell_1=0\}$ and $F':=\{\ell_2=0\}$, then $E_1=F_1\cup F'_1$. At least one of the two components has codimension $n$ in $E$. Say it is $F_1$, then $\mathrm{codim}_F (F_1)\leq n-1$. 

Together with the irreducible case, we conclude that at least one of the following holds:
\begin{enumerate}
    \item there exists a hyperplane $F\subset E$ such that $\mathrm{codim}_F (F_1)=n-1$;
    \item $E\subset \mathbb{C}^{k}\otimes\mathbb{C}^{l}$ such that $k+l\leq n+3$ and $k,l\geq 2$.
\end{enumerate}

Using induction on $\mathrm{dim}(E)$, we conclude.
\end{proof}

\subsection{Case $\mathrm{codim}(E_2)=1$}\label{E2codim1Section}
If $\mathrm{codim}(E_2)=1$, then there must exist an irreducible polynomial $P$ of degree $k\leq 3$ dividing all $3\times 3$ minors of $E$. If $k=1$, then $E$ contains a hyperplane $\{P=0\}$ which has bounded rank 2. If $k=3$, by Lemma \ref{codim1lem} $E\subset \mathbb{C}^3\otimes\mathbb{C}^3$.  

When $k=2$, by Corollary \ref{codim1cor} we have $E\subset \mathbb{C}^4\otimes\mathbb{C}^4$, which suffices us to assume $E\subset A\ot B$ with $\mathrm{dim}(A)=\mathrm{dim}(b)=4$. The following proposition finds all such spaces up to changes of bases in $E,A$ and $B$. 

\begin{prop}\label{E2codim1prop}
Let $E\subset A\otimes B:=\BC^4\ot\BC^4$. If there exists an irreducible polynomial $S$ of degree 2 dividing all $3\times 3$ minors of $E$, then at least one of the following holds:
\begin{enumerate}
    \item $E$ has bounded rank 3;
    \item up to changes of bases in $E,A$ and $B$, $E$ is either skew-symmetric, or has the form a diagonal block matrix $\mathrm{diag}(X,X)$ where
    $$X=\begin{pmatrix}
    x_1 & x_2\\
    x_2 & x_3
    \end{pmatrix}\text{ or }
    \begin{pmatrix}
    x_1 & x_2\\
    x_3 & x_4
    \end{pmatrix}$$
    depending on the rank of $S$.
\end{enumerate}
\end{prop}

We defer the proof to \S \ref{sectionofproof}.

\begin{rem}\label{error}
In light of Proposition 1 in \cite{beauville2018introduction} and Corollary 6.8 in \cite{beilinson}, one might hope to weaken the hypothesis to the zero set of $\mathrm{det}(E)=0$ is a quadric hypersurface. However Proposition 1 is incorrect, as the following counter-example shows: let $E\subset \BC^4\ot\BC^4$ be defined as
$$E:=\begin{pmatrix}
x_1 & x_2 & 0 & 0\\
x_3 & x_4 & 0 & x_1\\
0 & 0 & x_1 & x_2\\
0 & 0 & x_3 & x_4\\
\end{pmatrix}.$$
Then $\mathrm{det}(E)=(x_1x_4-x_2x_3)^2$ but the line bundle morphism $E:\mathcal{O}_{\mathbb{P}^3}(-1)^4\rightarrow \mathcal{O}_{\mathbb{P}^3}^4$ fails to have constant rank on $\{x_1x_4-x_2x_3=0\}$.
\end{rem}

\subsection{Case $\mathrm{codim}(E_r)\leq n$}
A subspace $E\subset A\ot B$ is said to be \textbf{concise} if the associated tensor $T\in E^*\ot A\ot B$ is concise. Equivalently, there does not exist changes of bases in $A$ or $B$ such that any column or row of $E$ consists of only zero entries. This section studies upper bounds of $\aaa$ and $\bbb$ for concise spaces $E$ satisfying $\mathrm{codim}(E_r)\leq n$.

\begin{prop}\label{Ercodimnprop}
For any positive integer $r,n$, there exist positive integers $M_1,M_2$, such that if there exists a concise space $E\subset A\ot B:=\BC^{\aaa}\ot\BC^{\bbb}$ with $\mathrm{codim}(E_r)\leq n$, then at least one of the following holds:
\begin{enumerate}[(1)]
    \item $\aaa$ or $\bbb\leq M_1$;
    \item $\aaa,\bbb\leq M_2$;
    \item $\exists$ a hyperplane $F\subset E$ such that $\mathrm{codim}_F (F_r)\leq n-1$;
    \item $\exists 1\leq i\leq r$ such that $E=H+H'$ where $\mathrm{codim}(H'_{r-i})\leq n$ and $H'\subset \BC^i\ot B$ or $A\ot \BC^i$.
\end{enumerate}
\end{prop}

\begin{proof}
Proof by induction on $r$. For $r=1$, by Proposition \ref{segreprop} we can set $M_1=1$ and $M_2=n+1$.For $r\geq 2$ we divide the problem into different cases by the value of $\kappa$.

1. Case $\kappa=0$.

$\kappa=0$ if and only if $E$ is E1-generic. Since $E_r\neq E$, $\aaa,\bbb\geq r+1$. By Theorem \ref{1genthm}, $\aaa+\bbb\leq n+2r+1$.

~\\
2. Case $\kappa=1$. 

We can put $y^1_1=0$ by changing bases. Then the $(\aaa-1)\times(\bbb-1)$ submatrix consisting of entries in the last $\aaa-1$ rows and the last $\bbb-1$ columns is either 1-generic, or has $\kappa=1$ so we can put $y^2_2=0$. Repeat this procedure until the bottom right $(\aaa-k)\times(\bbb-k)$ submatrix is 1-generic. Then $E=\begin{pmatrix}
C_{k\times k} & *\\
* & D_{s\times t}
\end{pmatrix}$
where $D$ is 1-generic, $s=\aaa-k,t=\bbb-k$ and
$$C=\begin{pmatrix}
0 & * & * & \cdots & *& *\\
* & 0 & * & \cdots & *& *\\
* & * & 0 & \cdots & *& *\\
\vdots & \vdots & \vdots & \ddots & \vdots\\
* & * & * & \cdots& 0 & *\\
* & * & * & \cdots& * & 0\\
\end{pmatrix}.$$

If $k\geq r+1$, consider the submatrix $C'$ consisting entries in the first $r+1$ rows and the last $k-1$ columns of $C$:
$$C'_{(r+1)\times(k-1)}=\begin{pmatrix}
 * & * & \cdots & *& *\\
 0 & * & \cdots & *& *\\
 * & 0 & \cdots & *& *\\
 \vdots & \vdots & \ddots & \vdots\\
 * & * & \cdots& * & 0\\
\end{pmatrix}$$

By the definition of $\kappa$, all nonzero entries in the same row or column of $C$ are linearly independent. Therefore $C'$ is a codimension $r-1$ subspace of some 1-generic space in $\BC^{k-1}\ot\BC^{r+1}$. By 2 of Theorem \ref{1genthm}, all $(r+1)\times(r+1)$ minors of $C'$ determines a subvariety of codimension $\geq k-r-1$, so $k\leq n+r+1$.

Now to find upper bounds for $s$ and $t$. If $s=r$ and $t\geq r$, the submatrix consisting of entries in the last $s+1$ rows and the last $t+1$ columns is a codimension 1 subspace of a 1-generic space in $\BC^{s+1}\otimes\BC^{t+1}$. So by 2 of Theorem \ref{1genthm} again $t\leq n+r-1$. Similarly if $t=r$ then $s\leq n+r-1$.

If $s,t\geq r+1$, the submatrix consisting of entries in the last $s$ rows and the last $t+1$ columns is 1-generic. By 1 of Theorem \ref{1genthm}, $n\geq 2$ and $s+t\leq n+2r$.

To put everything together, either $\aaa\leq n+2r$, $\bbb\leq n+2r$ or $\aaa,\bbb\leq 2n+2r$.

~\\
3. Case $2\leq \kappa \leq \max\{M_1(r-1,n), M_2(r-1,n)\}$.

Claim: there exist $M_i(r,n,g),i=1,2$ such that if $E$ has $\kappa=g$ and satisfies the hypothesis of the proposition, then either $\aaa$ or $\bbb\leq M_1(r,n,g)$, or $\aaa,\bbb\leq M_2(r,n,g)$, or the condition (3) holds.

We will find $M_i=M_i(r,n,g)$ by induction on $g$. By the last case, we can set $M_1(r,n,1)=n+2r$ and $M_2(r,n,1)=2n+2r$. Assume claim is true for spaces of $\kappa<g$.

Write 
$E=\begin{pmatrix}
C_{k\times k} & *\\
* & D_{s\times t}
\end{pmatrix}$ 
such that $C$ has zeros on the diagonal and $D$ is 1-generic. Let $M(r,n,g-1):=\max\{M_1(r,n,g-1),M_2(r,n,g-1)\}$. If $k>2M(r,n,g-1)+1$, then the submatrix $C'$ consisting of entries in the first $M(r,n,g-1)+1$ rows and the last $M(r,n,g-1)+1$ columns of $C$ is a space of $\kappa=g-1$. However by the definitions of $M_i(r,n,g-1)$'s, all $(r+1)\times(r+1)$ minors of $C'$ determine of codimension $>n$ subset. Therefore $k\leq2M(r,n,g-1)+1$.

Now $s$ and $t$ has the same upper bound as the last case. So we can set $M_1(r,n,g):=2M(r,n,g-1)+r$ and $M_1(r,n,g):=2M(r,n,g-1)+r+n$ which proves the claim.

~\\
4. Case $\kappa\geq \max\{M_1(r-1,n), M_2(r-1,n)\}+1$.

Choose bases and possibly take transpose so that all entries in the top left $\kappa\times\kappa'$ block of $E$ are zero for some $1\leq \kappa'\leq \kappa$. So
\begin{equation}\label{kappablock}
E=\begin{pmatrix}
O_{\kappa\times\kappa'} & H\\
G & D_{s\times t}
\end{pmatrix}.
\end{equation}
We take the largest $\kappa'$ so that the submatrix $H$ is concise in $\BC^{\kappa}\otimes\BC^{t}$. By the definition of $\kappa$, $G$ is 1-generic.

Consider the $(r+1)\times(r+1)$ minors consisting of any single entry of $G$ and any $r\times r$ minor of $H$. We must have $\mathrm{codim}_H (H_{r-1})\leq n$ unless condition (3) holds. Since $\kappa\geq \max\{M_1(r-1,n), M_2(r-1,n)\}+1$, $t\leq M_1(r-1,n)$.

If $\kappa'\leq r$, then $\bbb\leq  M_1(r-1,n)+r$.

If $\kappa'\geq r+1>s$, consider the $(r+1)\times(r+1)$ minors that are a product of an $s\times S$ minor of $G$ and a $(r+1-s)\times (r+1-s)$ minor of $H$. Then either $\mathrm{codim}_G (G_{s-1})\leq n$ or $\mathrm{codim}_H (H_{r-s})\leq n$. The latter inequality implies condition (4) holds. The former inequality implies $\kappa'\leq n+s-1\leq n+r-1$, then $\bbb\leq n+r-1+M_1(r-1,n)$.

If $\kappa',s\geq r+1$, then $\mathrm{codim}_G (G_{r})\leq n$. By Theorem \ref{1genthm} $s+\kappa'\leq n+2r+1$. So $\bbb\leq n+r+M_1(r-1,n)$.

~\\
Since $M_i(r,n,g+1)\geq M_i(r,n,g)$ for $g\geq 0$, $M_i(r,n,g)$ takes the maximum at $g=g':=\max\{M_1(r-1,n), M_2(r-1,n)\}$. So we can put $M_1(n,r)=\max\{M_1(r,n,g'),n+r+M_1(r-1,n)\}$ and $M_2=M_1(r,n,g')$ which proves the proposition.
\end{proof}

\begin{cor}\label{E2codim2cor}
Let $E\subset A\ot B$ be concise and satisfy $\mathrm{codim}E_2=2$. Then at least one of the following holds:
\begin{enumerate}
    \item $\aaa$ or $\bbb\leq 6$;
    \item $\aaa,\bbb\leq 8$;
    \item $\exists$ a hyperplane $F\subset E$ such that $\mathrm{codim}_F F_2 \leq 1$;
    \item $E$ has bounded rank 2.
\end{enumerate}
\end{cor}
\begin{proof}
For $\kappa=0$ or $1$, by the proof of Proposition \ref{Ercodimnprop} either $\aaa$, $\bbb\leq 6$, or $\aaa,\bbb\leq 8$.

For $\kappa=2$ or $3$, put $E$ into the form \ref{kappablock}. If the condition (3) does not hold, $G$ and $H$ are both 1-generic. Then by Theorem \ref{1genthm}, $\aaa$ or $\bbb\leq 6$.

For $\kappa\geq 4$, $H$ must have bounded rank 1 and $t=1$. If $\kappa\leq 3$, then $\bbb\leq 5$. If $\kappa\geq 4$, $G$ must have bounded rank 1 and $s=1$, then $E$ has bounded rank 2.
\end{proof}

\section{Geometric Rank 3}\label{sec:gr3}

This section studies the structure of the set of tensors with geometric rank at most 3.

\begin{thm}\label{gr3thm}
A tensor $T\in A\ot B\ot C$ has geometric rank at most 3 if and only if one of the following conditions holds:
\begin{enumerate}
    \item $T(A^*)$, $T(B^*)$ or $T(C^*)$ is of bounded rank $3$, or
    \item $\mathrm{SR}(T)\leq 3$, or
    \item up to changes of bases $T=M_{\langle 2\rangle}$.
\end{enumerate}
    
If $T$ is primitive of geometric rank 3, then up to changes of bases and permutations of $A$, $B$ and $C$, it is either the matrix multiplication tensor $M_{\langle 2\rangle}$ or the tensor such that $T(A^*)$ is a space of $4\times 4$ skew-symmetric matrices of dimension $4,5$ or $6$.
\end{thm}

\begin{proof}
By (\ref{altdef}), $\mathrm{GR}(T)\leq 3$ if and only if at least one of the following three cases holds:
\begin{enumerate}[(i)]
    \item $\mathrm{codim} A^*_3=0$;
    \item $\mathrm{codim} A^*_2\leq 1$;
    \item $\mathrm{codim} A^*_1\leq 2$.
\end{enumerate}

Case (i): $\mathrm{codim} A^*_3=0\iff T(A^*)$ has bounded rank 3.

Case (ii): If $\mathrm{codim} A^*_2=0$, then $\mathrm{GR}(T)=2$, so $T(A^*)$, $T(B^*)$ or $T(C^*)$ is of bounded rank $2$.

When $\mathrm{codim} A^*_2=1$, according to the discussion in \S\ref{E2codim1Section} and Proposition \ref{E2codim1prop}, at least one of the following holds:
\begin{enumerate}
    \item $T=T'+T''$ where $T'(A^*)$ is a space of bounded rank 2 and $\mathrm{ml}_A(T'')=1$, so $T$ is not primitive. 
    \item $T(A^*),T(B^*)$ or $T(C^*)$ has bounded rank 3;
    \item up to changes of bases $T=M_{\langle 2 \rangle}$.
\end{enumerate}
By classification of $\mathcal{GR}_2$, any non-primitive tensor of $\GR=3$ is either compression or at least one of $T(A^*),T(B^*)$ and $T(C^*)$ has bounded rank 3.

Case (iii): By the discussion in \S\ref{E1section}, if there is a nonzero $2\times 2$ minor that is a product of 2 linear forms, $T$ is not primitive. If all nonzero $2\times2$ minors are irreducible, $T(A^*)\subset \mathbb{C}^2\otimes\mathbb{C}^3$ or $\mathbb{C}^3\otimes\mathbb{C}^2$, so has bounded rank 2.

By classification of primitive spaces of bounded rank 3 \cite{EH88}, if $T(A^*)$ is primitive spaces of bounded rank 3, then either $T(B^*)$ or $T(C^*)$ is $4\times 4$ skew-symmetric.
\end{proof}

By classifications of $\mathcal{GR}_r$ for $r=1,2$ and $3$, we summarize the following relations between geometric rank and slice rank:
\begin{cor}
\begin{enumerate}
    \item $\GR(T)=1\iff\SR(T)=1$.
    \item If $ml_A(T)$, $ml_B(T)$ or $ml_C(T)> 3$, then $\GR(T)=2\iff\SR(T)=2$.
    \item If at least one of $ml_A(T)$, $ml_B(T)$ and $ml_C(T)> 6$, or at least two of them $>4$, then $\GR(T)=3\iff\SR(T)=3$.
\end{enumerate}
\end{cor}

However we cannot draw any similar conclusion for $r\geq 4$. As a counter example, let $T\in\BC^m\ot\BC^m\ot\BC^m$ be defined as
$$T(A^*):=\begin{pmatrix}
0 &x_1 &x_2 & & & &\\
-x_1 &0 &x_3& & & &\\
-x_2 &-x_3&0& & & &\\
&&& x_4 & x_5 &\cdots & x_m\\
&&&   & x_4 &&\\
&&&   & &\ddots&\\
&&&   & &&x_4
\end{pmatrix}.$$
Then $T$ is a direct sum of the primitive tensor of geometric rank 2 and a compression tensor of geometric rank 2. So $\GR(T)=4$, $\SR(T)=5$, and $T$ is concise no matter how large $m$ is.

\section{Geometric Rank 4 and in General}\label{sec:gr4}
\begin{thm} \label{gr4thm}
If $T\in A\otimes B\otimes C:=\BC^{\aaa}\ot\BC^{\bbb}\ot\BC^{\ccc}$ is primitive of geometric rank 4, then either at least 2 of $\ml_A(T)$, $\ml_B(T)$ and $\ml_C(T)$ are at most 6, or all of them are at most 8.
\end{thm}

\begin{proof}
$\mathrm{GR}(T)\leq 4$ if and only if one of the following cases holds:
\begin{enumerate}[(1)]
    \item $A^*_4=A^*$;
    \item codim$(A^*_3)\leq1$;
    \item codim$(A^*_2)\leq2$;
    \item codim$(A^*_1)\leq3$;
    \item codim$(A^*_0)\leq4$;
\end{enumerate}

(1)$\iff T(A^*)$ has bounded rank 4. $T$ is primitive only if $T(A^*)$ is a primitive space of bounded rank 4. By \cite{Atk83}, if a primitive space of bounded rank 4 has size $n_1\times n_2$, then either $n_1\leq5$ and $n_2\leq 10$, $n_1\leq10$ and $n_2\leq 5$, or $n_1\leq6$ and $n_2\leq 6$. So either $\ml_B(T)\leq 5$, or $\ml_C(T)\leq 5$, or $\ml_B(T),\ml_C(T)\leq 6$.

~\\
(2)$\iff$ there exists an irreducible polynomial $P$ of degree $\geq 1$ dividing all $4\times 4$ minors of $T(A^*)$. 

(2.1) $\mathrm{deg}P=1$: by Lemma \ref{primitivelem} $T$ is not primitive.
 
(2.2) $\mathrm{deg}P=2$: 
By Lemma \ref{codim1lem2}, up to changes of bases the upper left $4\times4$ submatrix of $T(A^*)$ has determinant equal to $P^2$ and $P$ divides all $3\times 3$ minors of the submatrix. Proposition \ref{E2codim1prop} gives a classification of such $4\times 4$ matrix. Since the determinant does not vanish, the submatrix cannot have bounded rank 3, so the submatrix is either skew-symmetric or has the form $\mathrm{diag}(X,X)$.

(2.2.i) Case $\mathrm{diag}(X,X)$: write $T(A^*)$ as the block form:
$$T(A^*)=\begin{pmatrix}
X & 0 & E_1\\
0 & X & E_2\\
D_1 & D_2 & F
\end{pmatrix}$$
where $X$ has determinant $S$, and $E_i$ and $D_j$ are $2\times (\ccc-4)$ and $(\bbb-4)\times 2$ blocks. 

For $1\leq i\leq 2, 3\leq k\leq 4, 5\leq j,l\leq \ccc$, the minor $\Delta^{i34j}_{12kl}=\Delta^{ij}_{12}\Delta^{34}_{kl}$ is divisible by the irreducible quadratic polynomial $S$. Therefore either $S|\Delta^{ij}_{12}$ or $S|\Delta^{34}_{kl}$. By Lemma \ref{codim1lem}, either $D_1$ or $E_2$ can be put to 0 by adding first 2 rows or columns to the rest. By the same argument, either $D_2$ or $E_1$ can be put to 0.

If $D_1=D_2=0$ or $E_1=E_2=0$, $S$ divides $\Delta^{13ij}_{13kl}=(y^1_1)^2\Delta^{ij}_{kl}$ for $i,j,k,l\geq 5$. So $S$ divides all $2\times 2$ minors of $F$. By Lemma \ref{codim1lem} either $F\subset\mathbb{C}^2\otimes\mathbb{C}^2$ or $F$ has bounded rank 1. Therefore $\mathrm{ml}_B(T)$ or $\mathrm{ml}_C(T)\leq 6$ and $T$ is not concise.

If $D_1=E_1=0$ or $D_2=E_2=0$, without loss of generalities assume $D_1=E_1=0$. Consider the minors $\Delta^{1ijk}_{1lst}=y^1_1\Delta^{ijk}_{lst}$ for $i,j,k,l,s,t\geq 3$. So $S$ divides all $3\times 3$ minors of 
$G:=\begin{pmatrix}
X & E_2\\
D_2 & F
\end{pmatrix}$. 
By Lemma \ref{E2codim1prop} either $G\subset\mathbb{C}^4\otimes\mathbb{C}^4$ or $G$ has bounded rank 2. Therefore $\mathrm{ml}_B(T)\leq 6$ or $\mathrm{ml}_C(T)\leq 6$ and $T$ is not concise.

(2.2.ii) Case skew-symmetric: permute the first 4 rows and columns to put $T(A^*)$ into the following form
$$T(A^*)= \left(
    \begin{array}{c|c|c}
      \begin{matrix}
      x_1& 0\\0& x_1
      \end{matrix} & 
      \begin{matrix}
      a& d\\c& b
      \end{matrix}&E_1\\
      \hline
      \begin{matrix}
      b& -d\\-c& a
      \end{matrix} & 
      \begin{matrix}
      e& 0\\0& e
      \end{matrix}&E_2\\
      \hline
      D_1 &D_2 &F
    \end{array}
    \right).$$
    
Adding the first two rows and columns to the rest, so that $y^1_i,y^2_i,y^i_1,y^i_2$ do not contain $x_1$ in their expression, for all $i,j$. $S=x_1e-ab+cd$ divides all $4\times 4$ minors of $T(A^*)$. Restricting to the subspace $\{x_1=0\}$, then $S':=-ab+cd$ divides all $4\times 4$ minors of $T(A^*)|_{x_1=0}$. 

If $S'$ is irreducible, for $3\leq i,k\leq 4$ and $j,l\geq 5$, consider the minors $\Delta^{12ij}_{12kl}=\Delta^{12}_{kl}\Delta^{ij}_{12}$ of $T(A^*)|_{x_1=0}$. Similar to case (i), either $D_1|_{x_1=0}$ or $E_1|_{x_1=0}$ can be put 0. Without loss of generality assume $D_1|_{x_1=0}$.

Now working on $T(A^*)$, entries in $D_1$ are multiples of $x_1$. Then by adding multiples of first two rows to the last $m-4$ rows we can put $D_1=0$. $S$ dividing $\Delta^{12ij}_{12kl}=(x_1)^2\Delta^{ij}_{kl}$ for $i,j\geq 5,k,l\geq 3$ implies it divides all $2\times 2$ minors of the $(m-4)\times(m-2)$ block $(D_2\,F)$. So either $(D_2\,F)$ has bounded rank $2$ or $(D_2\,F)\subset \mathbb{C}^2\otimes\mathbb{C}^2$. If by changing bases $(D_2\,F)$ has nonzero entries only in the first 2 rows, $\ml_B(T)\leq 6$. Otherwise, by changing bases we can put all nonzero entries of $(D_2\,F)$ in its first 2 or 3 column. Then consider the $4\times 4$ minors involving one entry of $(D_2\,F)$ and $3\times 3$ minors from the first 4 rows of $T(A^*)$. By Proposition \ref{E2codim1prop}, either $\ml_C(T)\leq 6$, or $ml_B(T),ml_C(T)\leq 7$.

If $S'=-ab+cd$ is reducible, by changing bases we can put the block 
$\begin{pmatrix}
a & b\\
c & d
\end{pmatrix}$ as 
$\begin{pmatrix}
0 & b'\\
c' & d'
\end{pmatrix}$ and the same for $\begin{pmatrix}
b & -d\\
-c & a
\end{pmatrix}$. Then the upper left $4\times 4$ block of $T(A^*)$ becomes
$$\begin{pmatrix}
x_1& 0 &0& d'\\
0& x_1 &c'& b'\\
b'& -d'&e& 0\\
-c'& 0&0& e
\end{pmatrix}.$$
Then permuting rows and columns we get
$$\begin{pmatrix}
0& 0 &x_1& d'\\
0& 0 &-c'& b'\\
b'& -d'&e& 0\\
c'& x_1&0& e
\end{pmatrix}.$$

By the same argument, we can put $D_1=0$. Then all $3\times 3$ minors of $(D_2\,F)$ are divisible by $S$. By Proposition \ref{E2codim1prop} either $(D_2\,F)$ has bounded rank 2 or $(D_2\,F)\subset\mathbb{C}^4\otimes\mathbb{C}^4$. By the same argument as the previous case, either $\ml_B(T)\leq 6$, or $\ml_B(T),\ml_C(T)\leq 8$.

(2.3) $\mathrm{deg}P=3$: by Lemma \ref{codim1lem2}, either $P$ factors into linear forms so $T$ is not primitive, or $T(A^*)$ has bounded rank 4.

~\\
(3) By Proposition \ref{E2codim2cor}, if $T$ is primitive then either $\ml_B\leq 6$, $\ml_C(T)\leq 6$ or $\ml_B(T),\ml_C(T)\leq 8$.

~\\
(4) By Proposition \ref{segreprop}, if $T$ is primitive then $\ml_B(T)+\ml_C(T)\leq 6$.

~\\
(5)$\iff \mathrm{dim}(T(A^*))\leq 4 \Rightarrow \mathrm{SR}(T)\leq 4$.

Putting everything together, either $\ml_B(T)\leq 8$, or $\ml_C(T)\leq 8$, or $\ml_B(T),\ml_C(T)\leq 6$. Since geometric rank is invariant by permuting $A$, $B$ and $C$, we also have: 
\begin{itemize}
    \item either $\ml_A(T)\leq 8$, or $\ml_C(T)\leq 8$, or $\ml_A(T),\ml_C(T)\leq 6$;
    \item either $\ml_A(T)\leq 8$, or $\ml_B(T)\leq 8$, or $\ml_A(T),\ml_B(T)\leq 6$.
\end{itemize}
By inclusion-exclusion argument, we conclude the theorem.
\end{proof}

\begin{cor} If $\ml_A(T)$, $\ml_B(T)$ and $\ml_C(T)>8$, then $\mathrm{GR}(T)\leq 4$ if and only if either $\mathrm{SR}(T)\leq 4$, or up to changes of bases $T=T'+T''$ where $T'$ is the $3\times3\times3$ skew-symmetric tensor and $\SR(T'')= 2$.
\end{cor}

As a consequence of Proposition \ref{Ercodimnprop}, we draw a general conclusion for primitive tensors of geometric rank $r$.

\begin{thm}\label{grrthm}
For all $r$, there exists a positive integer $N_r$, such that if $T\in A\ot B\ot C$ is primitive of geometric rank $r$, then at least two of $\ml_A(T)$, $\ml_B(T)$ and $\ml_C(T)$ are at most $N_r$.
\end{thm}

\section{Geometric Rank of n-part Tensors}\label{section:npart}
Let $n>2$ and $A^{(i)}:=\BC^{m_i}$ for $i\leq n$. For a tensor $T\in A^{(1)}\otimes\cdots\otimes A^{(n)}$, let
$$\hat{\Sigma}_T^{A^{(1)}\cdots A^{(n-1)}}:=\{(x_1,\cdots,x_{n-1})\in A^{(1)*}\times\cdots\times A^{(n-1)*}\mid \forall x_n\in A^{(n)*}, T(x_1,\cdots,x_n)=0\}$$

The \textbf{geometric rank} of $T$ is defined to be $\mathrm{GR}(T):=\mathrm{codim} (\hat{\Sigma}_T^{A^{(1)}\cdots A^{(n-1)}})$.

Regard $T$ as a linear map $A^{(1)*}\rightarrow A^{(2)*}\ot\cdots\ot A^{(n-1)*}$. Define $A^{(1)*}_j:=\{x_1\in A^{(1)*}\mid  \mathrm{GR}(T(x_1))\leq j\}$ where $T(x_1)$ is an $(n-1)$-part tensor. For any $i\leq n$, $A^{(i)*}_j$ is defined similarly. Similar to the alternative definition (\ref{altdef}) for tripartite tensors, the following proposition gives an alternative definition of geometric rank for $n$-part tensors.

\begin{prop}\label{altdef2}
For $T\in A^{(1)}\otimes\cdots\otimes A^{(n)}$, $\GR(T)$ is invariant under the any permutation of $A^{(1)*},\cdots,A^{(n)*}$. For any $1\leq i\leq n$,  $\GR(T)=\min\{\mathrm{codim}(A^{(i)}_{j}+j)\mid 0\leq j\leq \min\{m_1,\cdots,\hat{m}_i,\cdots,m_n\}\}$.
\end{prop}

\begin{proof}
Consider the first projection $\pi:\hat{\Sigma}_T^{A^{(1)}\cdots A^{(n-1)}} \rightarrow A^{(1)*}$. For any $1\leq j\leq \min\{m_2,\cdots,m_n\}$ and $x_1\in A^{(1)*}_{j}\backslash A^{(1)*}_{j-1}$, the fiber 
\begin{align*}
    \pi^{-1}(x_1)&=\{(x_1,x_2,\cdots,x_{n-1})\mid \forall x_n, T(x_1,\cdots,x_n)=0\}\\
    &=\{x_1\}\times\{(x_2,\cdots,x_{n-1})\mid \forall x_n, T(x_1,\cdots,x_n)=0\} \\
    &=\{x_1\}\times\{(x_2,\cdots,x_{n-1})\mid \forall x_n, T(x_1)(x_2,\cdots,x_n)=0\}\\
    &=\{x_1\}\times\hat{\Sigma}_{T(x_1)}^{A^{(2)}\cdots A^{(n-1)}}
\end{align*}
Then $\mathrm{dim}(\pi^{-1}(x_1))=\mathrm{dim}(\hat{\Sigma}_{T(x_1)}^{A^{(2)}\cdots A^{(n-1)}})=m_2+\cdots+m_{n-1}-j$. And for $x_1\in A^{(1)*}_0$, $\pi^{-1}(x_1)=A^{(2)}\times\cdots\times A^{(n-1)}$ which has dimension $m_2+\cdots+m_{n-1}$. So
\begin{align*}
    \mathrm{dim}(\hat{\Sigma}_T^{A^{(1)})\cdots A^{(n-1)}}&=\max\{\mathrm{dim}(\pi^{-1}(A^{(1)}_{j}))\mid 0\leq j\leq \min\{m_2,\cdots,m_n\}\}\\
    &=\max\{\mathrm{dim}(A^{(1)}_{j})+\mathrm{dim}(\pi^{-1}(A^{(1)}_{j}))\mid 0\leq j\leq \min\{m_2,\cdots,m_n\}\}\\
    &=\max\{\mathrm{dim}(A^{(1)}_{j})+m_2+\cdots+m_{n-1}-j\mid 0\leq j\leq \min\{m_2,\cdots,m_n\}\}
\end{align*}

Therefore $\GR(T)=\mathrm{codim}(\hat{\Sigma}_T^{A^{(1)}\cdots A^{(n-1)}} )=\min\{\mathrm{codim}(A^{(1)}_{j}+j)\mid 0\leq j\leq \min\{m_2,\cdots,m_n\}\}$. This proves the case $i=1$ and implies $\GR(T)$ is invariant under any permutation of the last $n-1$ factors $A^{(2)*},\cdots,A^{(n)*}$. By definition $\GR(T)$ is invariant under any permutation of the first $n-1$ factors, so it is invariant under any permutation of all $n$ factors. And the cases when $i>1$ follow by permuting the factors.
\end{proof}

A linear subspace $E\subset A^{(1)}\otimes\cdots\otimes A^{(n)}$ has $\textbf{bounded geometric rank r}$ if every element has geometric rank at most $r$. 

\begin{prop}\label{prop:npart2}
Let $n\geq 3$ and $T\in A^{(1)}\otimes\cdots\otimes A^{(n)}$. For all $r< n$, $\mathrm{GR}(T)\leq r$ if and only if there exists $i$ such that $T(A^{(i)*})$ has bounded geometric rank $r$ as a space of $(n-1)$-part tensors.
\end{prop}

\begin{proof}
$\Leftarrow$ direction is obvious by Proposition \ref{altdef2}. Prove $\Rightarrow$ by induction on $n$.

The base case is $n=3$, then $\mathrm{GR}(T)\leq 1\iff\mathrm{SR}(T)\leq 1 \iff$ $\exists i$, $T(A_i)$ has bounded geometric rank $1$; $\mathrm{GR}(T)\leq 2\iff$ $\exists i$, $T(A_i)$ has bounded rank $2$, and matrix rank coincides with geometric rank for 2-tensors.

Assume the proposition is true for all $n<N$ and $T\in A^{(1)}\otimes\cdots\otimes A^{(N)}$. By Proposition \ref{altdef2}, $\mathrm{GR}(T)\leq r$ if and only if $\exists 0\leq k\leq r$, such that $\mathrm{codim}(A^{(1)*}_k)\leq r-k$. If $k=r$, then $A^{(1)*}_k=A^{(1)*}$ and $T(A^{(1)*})$ has bounded geometric rank $r$. 

If $k<r$, for $x_1\in A^{(1)*}_k$ consider $T(x_1)$ is as an $(N-1)$-part tensor. By assumption $\mathrm{GR}(T(x_1))\leq k$ if and only if $\exists i>1$, $T(x_1)(A^{(i)*})$ has bounded geometric rank $k$. Therefore we have 
$\forall x_i\in A^{(i)*},\forall x_1\in A^{(1)*}_k,\mathrm{GR}(T(x_1)(x_i))\leq k$

$\Rightarrow \forall x_i\in A^{(i)*},\{x_1\in A^{(1)*}\mid \mathrm{GR}(T(x_i)(x_1))\leq k\}\supset A^{(i)*}_k$

$\Rightarrow \forall x_i\in A^{(i)*},\mathrm{codim}\{x_1\in A^{(1)*}\mid \mathrm{GR}(T(x_i)(x_1))\leq k\}\leq r-k$

$\Rightarrow \forall x_i\in A^{(i)*},\mathrm{GR}(T)\leq k$

$\Rightarrow T(A^{(i)*})$ has bounded rank $k$.
\end{proof}

$T$ is said to have \textbf{partition rank 1} if there exists a partition $[n]=I\sqcup J$ such that $T=T_1\otimes T_2$ for some nonzero $T_1\in \bigotimes_{i\in I}A^{(i)}$ and $T_2\in\bigotimes_{j\in J}A^{(j)}$. Regard $T$ as a multilinear function $T: A^{(1)*}\times\cdots\times A^{(n)*}\rightarrow \mathbb{C}$, then $T$ has partition rank 1 if and only if $T$ is a product of two non-constant multilinear functions. The \textbf{partition rank} of $T$ is the smallest integer $r$ such that $T$ can be written as a sum of $r$ partition rank 1 tensors, denoted as $\mathrm{PR}(T)$. 

Partition rank was introduced in \cite{partition} as a more general version of slice rank. By definition $\GR(T)\leq\mathrm{PR}(T)\leq \SR(T)$. And \cite{CM21} showed that the partition rank is at most $2^{n-1}$ times of the geometric rank for $n$-part tensors.

\begin{prop}\label{prop:npart3}
For $n\geq 3$, $T\in A^{(1)}\otimes\cdots\otimes A^{(n)}$ has geometric rank 1 if and only if it has partition rank 1.
\end{prop}

\begin{proof}
By definition $\mathrm{PR}(T)=1$ implies $\GR(T)=1$. We prove the other direction by induction on $n$. For $n=3$, by Remark 2.6 of \cite{GL20} $\GR(T)=1$ if and only if $\SR(T)=1$,  and slice rank agree with partition rank for tripartite tensors.

Assume the statement is true for $n<N$ and $T\in A^{(1)}\otimes\cdots\otimes A^{(N)}$. By Proposition \ref{prop:npart2} there exists $i$ such that $T(A^{(i)})$ consists of $(N-1)$-part tensors with geometric rank at most $1$. Without loss of generality assume $i=N$.

Let $\{z_j\}_{j=1}^{m_N}$ be a basis of $A^{(N)*}$. By assumption $\mathrm{PR}(T(z_j))\leq 1$ for all $j$, so we can write $T(z_j)=:f_jg_j$ for some multilinear function $f_j,g_j$. By the definition of geometric rank, $\{T(z_j)=0,\forall j\}\subset A^{(1)}\times\cdots\times A^{(N-1)}$ has codimension 1. By possibly swapping $f_j$ and $g_j$, assume $\{f_j=0,\forall j\}$ has codimension 1. Thus $f_j$'s have a common factor of positive degree, denoted as $f$. Then we can write $f_j=:fh_j$ for some $h_j$, so $T(z_j)=fh_jg_j$. 

Say $f$ is a multilinear function defined on $\Pi_{j\in I}A^{(j)}$ for some $I\subset [N-1]$, then $h_jg_j$ is defined on $\Pi_{j\in [N-1]\backslash I}A^{(j)}$. Define $g:(\Pi_{j\in [N-1]\backslash I}A^{(j)})\times A^{(N)}\rightarrow \mathbb{C}$ by $g(x,z_j):=(h_jg_j)(x)$. Therefore $T=fg$ and has partition rank 1.
\end{proof}


\section{Proof of Proposition \ref{E2codim1prop}}\label{sectionofproof}
Before proving the proposition, we need the following lemma.

\begin{lem}\label{appendixlem}
Let $E\subset A\ot B:=\BC^2\ot \BC^2$ be a matrix of linear forms in variables $x_1,\cdots,x_{\ccc}$. Define
$$X_1:=\begin{pmatrix}
    x_1 & x_2\\
    x_2 & x_3
    \end{pmatrix},
X_2:=\begin{pmatrix}
    x_1 & x_2\\
    x_3 & x_4
    \end{pmatrix}$$
\begin{enumerate}
    \item If $\mathrm{det}E=\mathrm{det}X_1$, then $E=X_1$ up to changes of bases in $A$ and $B$.
    \item If $\mathrm{det}E=\mathrm{det}X_2$, then either $E=X_2$ or $E=X_2^t$ up to changes of bases in $A$ and $B$.
\end{enumerate}
\end{lem}

\begin{proof}[Proof of Proposition \ref{E2codim1prop}]
Say $\Delta_3=x_1S$, then the upper left $3\times 3$ submatrix must be of the form $x_1Z+U$ where $Z$ is a $3\times3$ matrix of complex numbers and $U$ has bounded rank rank 2. Therefore up to changes of bases, $U$ is either compression or skew-symmetric. Since $\Delta_3\neq0$, taking transpose if necessary, we can write the upper left $3\times 3$ submatrix as one of the following forms:

$$\mathbf{(i)}\begin{pmatrix}
y^1_1 & y^1_2 & y^1_3\\
y^2_1 & y^2_2 & y^2_3\\
0 & 0 & x_1
\end{pmatrix},
\mathbf{(ii)}\begin{pmatrix}
x_1 & 0 & y^1_3\\
0 & x_1 & y^2_3\\
y^3_1 & y^3_2 & y^3_3
\end{pmatrix},
\mathbf{(iii)}\begin{pmatrix}
x_1 & y^1_2 & y^1_3\\
-y^1_2 & x_1 & y^2_3\\
-y^1_3 & -y^2_2 & x_1
\end{pmatrix}$$

For the rest of the proof, we will discuss each of the above cases.

~\\\textbf{Case (i).} $S=\Delta_{2}=y^1_1y^2_2-y^1_2y^2_1$ is an irreducible quadratic polynomial, hence has Waring rank 3 or 4. Changing basis in $E$ we can write $S=x_1x_3-(x_2)^2$ or $x_1x_4-x_2x_3$ depending on rank$(S)$. By Lemma \ref{appendixlem} we can put the top left $2\times2$ block as the form $X_1$ or $X_2$.

$S$ divides $\Delta^{134}_{123}=x_1\Delta^{14}_{12}$ and $\Delta^{234}_{123}=x_1\Delta^{24}_{12}$, therefore $(y^4_1,y^4_2)\in\mathrm{span}\{(y^1_1,y^1_2),(y^2_1,y^2_2)\}$. Adding multiples of the first and the second row to the $4$-th, we can set $y^4_1$ and $y^4_2$ to zeros.

If $\Delta^{34}_{34}=0$, the right bottom $2\times 2$ submatrix has bounded rank 1, then $E$ has bounded rank 3.

Assume $\Delta^{34}_{34}\neq0$. Since $S$ divides $\Delta^{234}_{234}=y^2_2\Delta^{34}_{34}\neq 0$, $\Delta^{34}_{34}$ is a non-zero multiple of $S$, hence can be normalized to $S$. Apply Lemma \ref{appendixlem} again, $E$ has one of the following forms:
$$\begin{pmatrix}
x_1 & x_2 & y^1_3 & y^1_4\\
x_2 & x_3 & y^2_3 & y^2_4\\
0 & 0 & x_1 & x_2\\
0 & 0 & x_2 & x_3
\end{pmatrix}\text{ if rank}(S)=3,\text { or }
\begin{pmatrix}
x_1 & x_2 & y^1_3 & y^1_4\\
x_3 & x_4 & y^2_3 & y^2_4\\
0 & 0 & x_1 & y^3_4\\
0 & 0 & y^4_3 & x_4
\end{pmatrix}\text{ if rank}(S)=4$$
where $(y^3_4,y^4_3)=(x_2,x_3)$ or $(x_3,x_2)$. 

We consider separately the two subcases \textbf{(i.1)} rank(S) = 3 and \textbf{(i.2)} rank(S) = 4. And we further divide subcase \textbf{(i.2)} into two situations: \textbf{(i.2.1)} $(y^3_4,y^4_3)=(x_2,x_3)$, and \textbf{(i.2.2)} $(y^3_4,y^4_3)=(x_3,x_2)$.

\textbf{Subcase (i.1).} Assume rank$(S)=3$. Write $y^1_4=l_2+l'_2$, $y^2_3=l_3+l'_3$, $y^2_4=l_4+l'_4$, and $\Delta^{123}_{234}=(l+l')S$ where $l,l_i\in\mathrm{span}\{x_1,x_2,x_3\}$ and $l',l'_i\in\mathrm{span}\{x_4,\cdots,x_m\}$. Then  
$$l'S=l'(x_1x_3-(x_2)^2)=(x_3l'_2-x_2l'_4)x_1+(x_2l'_3-x_3l'_1)x_2.$$

Comparing terms that are multiples of $(x_1)^2$, we see $l'=0$, which forces all $l'_i=0$, so $y^i_j\in\mathrm{span}\{x_1,x_2,x_3\}$. Adding multiples of the first two rows and columns to the last two rows and columns, we can put $y^1_4=y^2_3=0$. 

Write $l=a_1x_1+a_2x_2+a_3x_3$. Then
$$(a_1x_1+a_2x_2+a_3x_3)(x_1x_3-(x_2)^2)=lS=\Delta^{123}_{234}=-x_2x_3y^1_3-x_1x_2y^2_4.$$
Comparing the terms of multiples of $(x_1)^2x_3, x_1(x_3)^2$ and $(x_2)^3$, we see all $a_i=0$ so $l=0$. Comparing the coefficients of the rest cubic monomials, we obtain $a_1+b_3=0$ and $a_2=a_3=b_1=b_2=0$. If $a_1=0$, $E$ has the form diag$(X_1,X_1)$ where $X_1=\begin{pmatrix}
x_1 & x_2\\
x_2 & x_3
\end{pmatrix}$.
If $a_1\neq 0$, multiply $a_1$ to the first two column and the last two rows, subtract the 1st row from the 3rd, and add the 4th row to the 2nd, then $E$ again has the form diag$(X_1,X_1)$.

\textbf{Subcase (i.2).} Assume rank$(S)=4$.

\textbf{(i.2.1).} If $(y^3_4,y^4_3)=(x_2,x_3)$: using the notations in (i.1), write $y^i_j=l_k+l'_k$ and $\Delta^{123}_{234}=(l+l')S$. Then $l'S=x_2(l'_3x_2-l'_4x_1)-x_4(l'_1x_2-l'_2x_1)$. Comparing terms we see all $l'_i=0$. Then by adding multiples of the first two columns to the third and fourth, then adding multiples of the first two rows to the first and second, we can make $y^1_3\in\mathrm{span}\{x_2,x_4\},y^2_3\in\mathrm{span}\{x_2\},y^1_3\in\mathrm{span}\{x_1,x_2,x_4\}$.

Since $\Delta^{123}_{234}=lS=x_2(y^2_3x_2-y^2_4x_1)-x_4(y^1_3x_2-y^1_4x_1)$, writing $y^i_j$ into linear combinations of $x_k$'s, we see that there is no $x_2x_3$. Thus $l=0$, then comparing terms of the above equation, we have $y^1_3=y^2_3=0$. By $\Delta^{123}_{134}$ and $\Delta^{123}_{234}$, $(y^1_4,y^2_4)^t\in\mathrm{span}\{(x_1,x_3)^t,(x_2,x_4)^t\}$, so we can set $(y^1_4,y^2_4)^t=0$ by adding multiples of the first two columns to the fourth.

Therefore, by changing bases $E$ has the form $\mathrm{diag}(X_2,X_2)$ where 
$X_2=\begin{pmatrix}
x_1 & x_2\\
x_3 & x_4
\end{pmatrix}$.

\textbf{(i.2.2).} If $(y^3_4,y^4_3)=(x_3,x_2)$: using the notations in (i.1), write $y^i_j=l_k+l'_k$ and $\Delta^{123}_{234}=(l+l')S$. Then $l'S=x_2(l'_3x_3-l'_4x_1)-x_4(l'_1x_3-l'_2x_1)$. Comparing terms we see $l'=-l'_3=l'_2$ and $l'_1=l'_4=0$. Then by adding multiples of the first two columns to the third and fourth, then adding multiples of the first two rows to the first and second, we can write $y^1_4=l',y^2_3=-l',y^1_3=\sum_{i=1}^4 a_ix_i,y^2_4=\sum_{i=1}^4 d_ix_i$.

$S$ divides other $3\times 3$ minors, which implies $y^1_3=y^2_4=0$. Therefore the only nonzero entries of $E$ are in the upper left $4\times4$ block of the form:
$$\begin{pmatrix}
x_1 & x_3 & 0 & l'\\
x_2 & x_4 & -l' & 0\\
0 & 0 & x_1 & x_2\\
0 & 0 & x_3 & x_4
\end{pmatrix}.$$
Swapping the first two rows and the last two rows, then multiply -1 to the first 2 rows. $E$ becomes skew-symmetric.

~\\\textbf{Case (ii).} Here $S=x_1y^3_3-y^2_3y^3_2-y^1_3y^3_1$. Modify $y^1_j,y^2_j,y^i_1,y^i_2$,
$3\leq i,j \leq 4$
such that their expressions (as linear forms in $x_i$'s) do not contain $x_1$.

If $y^3_3=0$ and $\exists i,j>2$, such that $y^i_j\neq 0$, we may change bases such that $y^3_3\neq 0$, so we have two cases $y^3_3\neq 0$ or
$y^i_j=0$ for all $i,j>2$.

If $y^3_3\neq 0$, then consider $\Delta^{12i}_{12j}=x_1(x_1y^i_j-y^i_1y^1_j-y^i_2y^2_j)$. We obtain
$y^i_j=c^i_j y^3_3$ for all $i,j>2$ and constants $c^i_j$. Changing bases again, we may set $y^3_4=y^4_3=0$ and $y^4_4=y^3_3$ or 0. There are 3 subcases: \textbf{(ii.1)} $y^4_4=y^3_3\neq0$, \textbf{(ii.2)} $y^4_4=0,y^3_3\neq0$, and \textbf{(ii.3)} $y^3_3=y^4_4=0$.

\textbf{Subcase (ii.1).} Assume $y^4_4=y^3_3\neq0$. Then $\Delta^{\rho34}_{\rho34}=y^3_3(x_1y^3_3-y^3_\rho y^\rho_3-y^4_\rho y^\rho_4),\forall \rho=1,2$. Together with $\Delta^{12i}_{12i}$'s we get
$$y^3_1 y^1_3+y^4_1 y^1_4=y^3_2 y^2_3+y^4_2 y^2_4=y^3_1y^1_3+y^3_2y^2_3=y^4_1y^1_4+y^4_2y^2_4$$
Hence $y^3_1 y^1_3=y^4_2y^2_4$, $y^4_1 y^1_4=y^3_2y^2_3$.

$\Delta^{\rho34}_{\sigma34}=y^3_3(-y^3_\sigma y^\rho_3-y^4_\sigma y^\rho_4)=0$ for $(\rho,\sigma)=(0,1)$ or $(1,0)$, and $\Delta^{12i}_{12j}=0$ for $i\neq j$.  We get
$$y^3_1y^1_4+y^3_2y^2_4=y^4_1y^1_3+y^4_2y^2_3=y^3_1y^2_3+y^4_1y^2_4=y^3_2y^1_3+y^4_2y^1_4=0$$

In other words, denoting $Q:=y^3_1 y^1_3+y^4_1 y^1_4$, the following equations hold:
$$\begin{pmatrix}
y^3_1 & y^3_2\\
y^4_1 & y^4_2
\end{pmatrix}
\begin{pmatrix}
y^1_3 & y^1_4\\
y^2_3 & y^2_4
\end{pmatrix}
=\begin{pmatrix}
y^1_3 & y^1_4\\
y^2_3 & y^2_4
\end{pmatrix}
\begin{pmatrix}
y^3_1 & y^3_2\\
y^4_1 & y^4_2
\end{pmatrix}
=\begin{pmatrix}
Q & 0\\
0 & Q
\end{pmatrix}$$

Then by changing bases $E$ equals to the matrix whose upper left $4\times 4$ block is one of the following, and all other entries are zeros:
$$\begin{pmatrix}
x_1& 0& a& d\\
0& x_1& c& b\\
b& -d& y^3_3& 0\\
-c& a& 0& y^3_3
\end{pmatrix}$$
for some linear forms or zeros $a,b,c,d,y^3_3$.

\textbf{Subcase (ii.2).} Assume $y^4_4=0,y^3_3\neq0$. Since $\Delta_4\neq 0$, there exist $\rho,\sigma=1,2$, such that $y^4_{\sigma}$ and $y^{\rho}_4 \neq 0$. Then $\Delta^{124}_{\sigma 34}$ and $\Delta^{\rho 34}_{124}$ implies $\Delta^{12}_{34}=\Delta^{34}_{12}=0$. Then change bases in the first two rows and columns, we get:
$$\begin{pmatrix}
c_1x_1 & c_2 x_1 & y^1_3 & y^1_4 \\
c_3 x_1 & c_4 x_1 & 0 & 0 \\
y^3_1 & 0 & y^3_3 & 0 \\
y^4_1 & 0 & 0 &0 \\
\end{pmatrix}$$
for some constants $c_i$. $\Delta^{123}_{234}=-c_4x_1y^1_4y^3_3$ implies $c_4=0$. Then $\Delta^{123}_{123}=-c_2c_3(x_1)^2y^3_3$ implies either $c_2$ or $c_3=0$, contradicting the hypothesis $\Delta_3\neq 0$.

\textbf{Subcase (ii.3).} Assume $y^3_3=y^4_4=0$. As $S=y^3_1y^1_3+y^3_2y^2_3$ is irreducible, $y^3_1,y^3_2$ are linearly independent, and so are $y^1_3,y^2_3$. Choose bases such that $y^1_1$ and $y^2_2$ are not necessary $x_1$, and $y^3_1=x_1,y^3_2=x_2$. Since rank$(S)\geq 3$, at least one of $y^1_3,y^2_3$ is linearly independent with $x_1,x_2$. Without loss of generality assume $x_1,x_2,y^1_3$ are linearly independent, then choose bases such that $y^1_3=x_3$:
$$\begin{pmatrix}
y^1_1 & 0 & x_3 & y^1_4  \\
0 & y^1_1 & y^2_3 & y^2_4\\
x_1 & x_2 & 0   & 0 \\
y^4_1 & y^4_2 & 0   & 0\\
\end{pmatrix}.$$

If $\Delta^{12}_{34}=0$, 
$\begin{pmatrix}
 x_3 & y^1_4 \\
 y^2_3 & y^2_4\\
\end{pmatrix}$ has bounded rank 1 so we can set either the fourth column to zero (then $E\subset\mathbb{C}^4\otimes\mathbb{C}^3$), or $y^2_3=y^2_4=0$ (then $\Delta^{123}_{123}$ is a product of linear forms, contradicting to irreducibility of $S$). 

If $\Delta^{12}_{34}\neq0$, by linear independence of $y^1_3=x_3$ and $y^2_3$, and $\Delta^{12}_{34}$ is a nonzero multiple of $S=x_1y^1_3+x_2y^2_3$, we can normalize the fourth column such that $(y^1_4,y^2_4)^t=(x_2,-x_1)^t$ and $y^{\rho}_j=0,\forall j>4, \rho=1,2$. By the same argument, we can set $(y^4_1,y^4_2)=(y^2_3,-x_3)^t$ and $y^i_{\sigma}=0,\forall i>4, \sigma=1,2$.

Then $E$ has the form:
$$\begin{pmatrix}
y^1_1 & 0 & x_3 & x_2\\
0 & y^1_1 & y^2_3 & -x_1\\
x_1 & x_2 & 0   & 0 \\
y^2_3 & -x_3 & 0  & 0\\
\end{pmatrix},$$
which is skew-symmetric after permuting rows and columns.

~\\\textbf{Case (iii).} Here $S=x_1^2+(y^1_2)^2+(y^1_3)^2+(y^2_3)^2$ is irreducible, so rank$(S)>2$. We consider two subcases by whether $x_1,y^1_2,y^1_3,y^2_3$ are linearly independent.

\textbf{Subcase (iii.1).} Assume $x_1,y^1_2,y^1_3,y^2_3$ are linearly independent. We can choose basis of $E$ such that $y^1_2=x_2,y^1_3=x_3,y^2_3=x_4$. In order that $S$ divides all $3\times 3$ minors, $(y^1_4,y^2_4,y^3_4)^t$ must be a linear combination of $(x_1,-x_2,-x_3)^t$,  $(x_2,x_1,-x_4)^t$, $(x_3,x_4,x_1)^t$, and $(x_4,-x_3,x_2)^t$. Changing the basis we can put $(y^1_4,y^2_4,y^3_4)^t=(x_4,-x_3,x_2)^t$. By the same argument, $(y^4_1,y^4_2,y^4_3)=(-x_4,+x_3,-x_2)$. Consider the $3\times3$ minors involving $y^4_4$, we see $y^4_4=-x_1$.

Hence $E$ has the form:
\begin{equation*}
    \begin{pmatrix}
    x_1 & x_2 & x_3 & x_4 \\
    -x_2 & x_1 & x_4 & -x_3\\
    -x_3 & -x_4 & x_1 & x_2\\
    -x_4 & x_3 & -x_2 & x_1\\
    \end{pmatrix}.
\end{equation*}
Then $E$ is the the complex quaternion algebra $\mathrm{span}_{\mathbb{C}}\{\mathds{1},I,J,K\}/(I^2+\mathds{1},J^2+\mathds{1},K^2+\mathds{1},IJK+\mathds{1})$ and the associated tensor of $E$ is the structure tensor of the complex quaternion. Since the complex quaternion algebra is isomorphic to the matrix algebra $\mathrm{Mat}_{2\times2}$, their structure tensors equal up to changes of bases. So $E$ equals to $M_{\langle2\rangle}$ up to changes of bases in $E,A$ and $B$.

\textbf{Subcase (iii.2).} Assume $x_1,y^1_2,y^1_3,y^2_3$ are linearly dependent. The irreducibility of $S$ implies three of them are linearly independent. $x_1\neq0$ since $\Delta_3\neq 0$. By changing bases assume $y^1_2=x_2,y^1_3=x_3,y^2_3=a_1x_1+a_2x_2+a_3x_3$ for $a_i\in\mathbb{C}$. Then $S=x_1^2+x_2^2+x_3^2+(a_1x_1+a_2x_2+a_3x_3)^2$. If $\mathrm{dim}\langle x_i,y^i_4\mid i=1,2,3\rangle\geq 5$, the submatrix consisting of the first 3 rows is a subspace of a 1-generic space of codimension $\leq 2$, then Theorem \ref{1genthm} implies contradiction. So $\mathrm{dim}\langle x_i,y^i_4\mid i=1,2,3\rangle\leq 4$.

Adding first 3 columns to the 4th, we can set $y^1_4=0$ or $x_4$. Write $y^2_4=\sum_i b_ix_i$ and $y^3_4=\sum_i c_ix_i$, $\Delta^{123}_{124}=LS$,$\Delta^{123}_{134}=MS$ and $\Delta^{123}_{224}=NS$ for some linear forms $L=\sum_i l_ix_i,M,N$.

If $y^1_4=x_4$: 
\begin{align*}
    \Delta^{123}_{124}=&((a_1b_4+c_4)x_1^2 + (a_2+c_4)x_2^2 +(1+a_3b_4)x_1x_3 + (a_2b_4+a_1)x_1x_2 + (a_3-b_4)x_2x_3)x_4 \\
    &+(c_1+a_1b_1)x_1^3+ c_2x_2^3 + (c_3-b_2)x_2^2x_3 - b_3x_2x_3^2 + (c_1+a_2b_2)x_1x_2^2\\
    &+ (c_2+a_1b_2+a_2b_1)x_1^2x_2 + (c_3+a_1b_3+a_3b_1)x_1^2x_3 + a_3b_3x_1x_3^2 +(a_2b_3+a_3b_2-b_1)x_1x_2x_3.
\end{align*}

Note that there is no $x_3^2x_4$ in $\Delta^{123}_{124}$. This implies either $a_3^2+1=0$ or $l_4=0$. If $l_4=0$, those terms divisible by $x_4$ have the sum zero: 
$$(a_1b_4+c_4)x_1^2 + (a_2+c_4)x2^2 +(1+a_3b_4)x1x3 + (a_2b_4+a_1)x_1x_2 + (a_3-b_4)x_2x_3=0$$
which implies $a_3^2+1=0$. Hence $a_3^2+1=0$ no matter if $l_4=0$.

There is no $x_2^2x_4$ in $\Delta^{123}_{134}$, thus by the same argument, we must have $a_2^2+1=0$. 

Compare the coefficients of $x_2^3$ in equality $\Delta^{123}_{124}=LS$ and $x_3^3$ in $\Delta^{123}_{134}=MS$, we get 
$$c_2=l_2(1+a_2^2)=0\text{ and }-b_3=m_3(1+a_3^2)=0.$$ 
Compare the coefficients of $x_2^2x_4$ in $\Delta^{123}_{124}=LS$ and $x_3^2x_4$ in $\Delta^{123}_{134}=MS$, we get $$(a_2+c_4)=l_4(1+a_2^2)=0 \text{ and } (a_3-b_4)=m_4(1+a_3^2)=0.$$
Compare the coefficients of $x_2x_3x_4$ in $\Delta^{123}_{124}=LS$ and $\Delta^{123}_{134}=MS$, we get
$$2a_2a_3l_4=(a_3-b_4)=0\text{ and }2a_2a_3m_4=a_2+c_4=0.$$
Therefore $l_4=m_4=0$. Then the coefficients of every monomial divisible by $x_4$ in $\Delta^{123}_{124}$ and $\Delta^{123}_{134}$ equals zero. We get $a_1=a_2/a_3$ from $\Delta^{123}_{124}$ but $a_1=a_2a_3$ contradicting $a_3^2=-1$.

If $y^1_4=0$: since there is no $x_3^2x_4$ in $\Delta^{123}_{124}$, either $a_3^2+1=0$ or $l_4=0$. 

If $l_4=0$, then the coefficients of every monomial divisible by $x_4$ in $\Delta^{123}_{124}$ equal zero, which implies $b_4=c_4=0$. Therefore there is no $x_4$ appearing in the first 3 rows, and by the same argument $x_4$ does not appear in the first 3 columns. If $\exists i,j>3$, such that $y^i_j\notin\mathrm{span}\{x_1,x_2,x_3\}$, then we can change basis in $E$ to set $y^i_j=x_4$. Write $\Delta^{12i}_{12j}=x_4(x_1^2+x_2^2)+p(x_1,x_2,x_3)$ for some polynomial $p$. $S=S(x_1,x_2,x_3)$ dividing $\Delta^{12i}_{12j}\neq 0$ implies that $S$ divides $x_1^2+x_2^2$, contradicting to the irreducibility of quadratic polynomial $S$. If there is no such $y^i_j$, then $\mathrm{dim}(E)=3$.

Therefore $a_3^2+1=0$. And by the same argument, since there is no $x_2^2x_4$ in $\Delta^{123}_{134}$, $a_2^2+1=0$. Compare the coefficients of $x_2^2x_4$ in $\Delta^{123}_{124}=LS$ and $x_3^2x_4$ in $\Delta^{123}_{134}=MS$, we get $c_4=b_4=0$.

Then by the same argument as in the case $l_4=0$ we obtain $\mathrm{dim}(E)=3$.
\end{proof}

\bibliographystyle{amsplain}
\bibliography{geomrk.bib}
\end{document}